\documentclass[letterpaper,12pt]{article}
\usepackage[total={6in, 8in}]{geometry}

\usepackage{amsmath}
\usepackage{amssymb}
\usepackage{amsthm}

\usepackage{mathtools}

\usepackage{array}
\usepackage{multirow}
\usepackage{bbm}
\usepackage{enumitem}
\usepackage{float}
\usepackage{tikz}
\usepackage{thmtools}
\usepackage{thm-restate}

\usepackage{algorithm}
\usepackage{algorithmic}

\usepackage{xparse}

\usepackage[x11names]{xcolor}

\usepackage{hyperref}
\hypersetup{%
    colorlinks = true,
    urlcolor   = cyan,
    linkcolor  = magenta,
    citecolor  = blue%
}

\usepackage{thm-restate}

\theoremstyle{plain}
\newtheorem{theorem}{Theorem}[section]

\newtheorem{lemma}[theorem]{Lemma}
\newtheorem{example}[theorem]{Example}
\newtheorem{corollary}[theorem]{Corollary}

\theoremstyle{definition}
\newtheorem{definition}[theorem]{Definition}

\theoremstyle{remark}

\newtheorem*{remark*}{Remark}

\usepackage[capitalise]{cleveref}
\usepackage{amsmath,amsfonts,amssymb,bm,bbm,pifont}

\usepackage{amsmath}
\usepackage{xparse}
\DeclareFontFamily{U}{ntxmia}{}
\DeclareFontShape{U}{ntxmia}{m}{it}{<-> ntxmia }{}
\DeclareFontShape{U}{ntxmia}{b}{it}{<-> ntxbmia }{}
\DeclareSymbolFont{lettersA}{U}{ntxmia}{m}{it}
\SetSymbolFont{lettersA}{bold}{U}{ntxmia}{b}{it}

\ExplSyntaxOn
\NewDocumentCommand{\varmathbb}{m}
 {
  \tl_map_inline:nn { #1 }
   {
    \use:c { varbb##1 }
   }
 }
\cs_new_protected:Nn \__mathbb_define:Nn
 {
  \DeclareMathSymbol{#1}{\mathord}{lettersA}{#2}
 }
\cs_generate_variant:Nn \__mathbb_define:Nn {ce}
\tl_map_inline:nn { ABCDEFGHIJKLMNOPQRSTUVWXYZ }
 {
  \__mathbb_define:ce { varbb#1 } { \int_eval:n { `#1+67 } }
 }
\tl_map_inline:nn { abcdefghijklmnopqrstuvwxyz }
 {
  \__mathbb_define:ce { varbb#1 } { \int_eval:n { `#1+61 } }
 }
\DeclareMathSymbol{\varbbimath}{\mathord}{lettersA}{'270}
\DeclareMathSymbol{\varbbjmath}{\mathord}{lettersA}{'271}
\ExplSyntaxOff

\def\norm#1{\lVert#1\rVert}
\def\inner#1{\left\langle#1\right\rangle}

\def\bignorm#1{\left\lVert#1\right\rVert}

\def\bigopen#1{\left(#1\right)}
\def\bigset#1{\left\{#1\right\}}

\newcommand{\dd}{\mathrm{d}}
\newcommand{\dom}{\mathrm{dom} \, }

\newcommand{\vxtilde}{\widetilde{\boldsymbol{x}}}
\newcommand{\vytilde}{\widetilde{\boldsymbol{y}}}
\newcommand{\vztilde}{\widetilde{\boldsymbol{z}}}

\newcommand{\vzhat}{\widehat{\boldsymbol{z}}}

\def\1{\bm{1}}

\def\rvz{{\mathbf{z}}}

\def\va{{\bm{a}}}
\def\vb{{\bm{b}}}

\def\vg{{\bm{g}}}

\def\vw{{\bm{w}}}
\def\vx{{\bm{x}}}
\def\vy{{\bm{y}}}
\def\vz{{\bm{z}}}

\def\mM{{\bm{M}}}

\DeclareMathAlphabet{\mathsfit}{\encodingdefault}{\sfdefault}{m}{sl}
\SetMathAlphabet{\mathsfit}{bold}{\encodingdefault}{\sfdefault}{bx}{n}

\def\gL{{\mathcal{L}}}

\def\gO{{\mathcal{O}}}

\def\sZ{{\mathbb{Z}}}

\def\oA{{\varmathbb{A}}}
\def\oB{{\varmathbb{B}}}

\def\oS{{\varmathbb{S}}}
\def\oT{{\varmathbb{T}}}

\newcommand{\R}{\mathbb{R}}

\title{Nesterov Acceleration with Operator Decomposition%
\footnote{
JL was supported in part by the Krishna Kolluri Graduate Fellowship at Stanford University. 
EKR was supported by the NSF grant CCF-2504627. CY was supported by the 2025 KAIST-U.S.\ Joint Research Collaboration Open Track Project for Early-Career Researchers supported by the International Office at the Korea Advanced Institute of Science and Technology (KAIST), and the National Research Foundation of Korea (NRF) grant (No.\ RS-2023-00211352) funded by the Korean government (MSIT).}%
}
\author{Jaewook Lee%
\thanks{Department of Electrical Engineering, Stanford University
(\texttt{jwl99@stanford.edu}).}
\and Ernest K. Ryu%
\thanks{Department of Mathematics,
University of California, Los Angeles
(\texttt{eryu@math.ucla.edu}).}
\and Chulhee Yun%
\thanks{Kim Jaechul Graduate School of AI, KAIST
(\texttt{chulhee.yun@kaist.ac.kr}).}
}

\begin{document}
\maketitle

\begin{abstract}
We propose Nesterov acceleration with Operator Decomposition (NOD), which extends Nesterov's accelerated gradient descent (NAG) from smooth strongly convex optimization to the broader setting of strongly monotone, Lipschitz operators. The key insight is to  decompose the operator into cyclically monotone and monotone components, with the Asplund decomposition providing the tightest such representation, and to have the algorithm utilize the decomposed oracles. NOD and its analysis subsume the classical theory of Nesterov acceleration and yield an iteration complexity for finding an $\epsilon$-accurate solution of
\[
\Theta\left(\sqrt{\frac{L_{\phi}}{\mu} + \frac{L_{\oS}^2}{\mu^2}} \,\log \frac{1}{\epsilon}\right),
\]
where $\mu$ is the strong monotonicity parameter, $L_\phi$ is the Lipschitz constant of the cyclically monotone component, and $L_{\oS}$ is the Lipschitz constant of the (possibly acyclic) monotone remainder.
\end{abstract}

\clearpage


\section{Introduction}
\label{sec:1}

Consider the monotone inclusion problem
\begin{align*}
    \underset{\vz \in \R^{d}}{\mathrm{find}} \quad \oT (\vz) = \bm{0},
\end{align*}
where $\oT\colon \R^{d} \rightarrow \R^{d}$ is a maximal, strongly monotone, and Lipschitz operator.
This general problem class subsumes the widely studied classes of minimization problems,
\begin{align*}
    \min_{\vx\in \R^{d}} f(\vx),
\end{align*}
where $f \colon \R^{d}\to\R$ is strongly convex and has Lipschitz gradients,
and minimax optimization problems,
\begin{align*}
    \min_{\vx \in \R^{d_x}} \max_{\vy \in \R^{d_y}} \mathcal{L} (\vx, \vy),
\end{align*}
where $\mathcal{L} \colon \R^{d_x} \times \R^{d_y} \to \R$ is strongly-convex-strongly-concave (SCSC) and has Lipschitz gradients.
Finding the optimal point of the minimization problem is equivalent to finding a zero of $\oT(\vx) = \nabla f(\vx)$, and finding the (pure) Nash equilibrium of the minimax optimization problem is equivalent to finding a zero of the saddle operator
\[\oT (\vx, \vy) = \begin{bmatrix}
    \nabla_{\vx} \mathcal{L} (\vx, \vy) \\
    - \nabla_{\vy} \mathcal{L} (\vx, \vy)
\end{bmatrix}.\]

In 1970, Asplund \cite{asplund70} showed that a single-valued maximal monotone operator $\oT \colon \R^{d} \rightarrow \R^{d}$ admits an \textit{Asplund decomposition} of the form
\begin{align*}
    \oT &= \underbrace{\partial \phi}_{\text{cyclically monotone}} + 
    \underbrace{\oS}_{\text{acyclic monotone}}
\end{align*}
where $\partial \phi$ is a cyclically monotone operator, i.e., the subdifferential operator of a convex function $\phi \colon\mathbb{R}^d\rightarrow\mathbb{R}$ \cite{rockafellar66}, and $\oS$ is an \textit{acyclic monotone} operator, which is a monotone operator with only affine cyclically monotone components.
The canonical example is the saddle gradient operator of a smooth convex-concave saddle function with bilinear coupling $\mathcal{L} (\vx, \vy) = g(\vx) + \vx^{\top} \mM \vy - h(\vy)$, for which we have $\phi(\vx, \vy) = g(\vx) + h(\vy)$ and $\oS(\vx, \vy) = [\mM \vy ; \ {- \mM^{\top} \vx}]$ (as discussed further in Example~\ref{ex:2}).

However, the Asplund decomposition can also admit non-bilinear coupling terms.
Specifically, consider $\mathcal{L}  \colon \mathbb{R}^2\rightarrow\mathbb{R}$ defined by
\begin{align*}
\mathcal{L} (x, y) = x^2 - y^2 + \sin x \sin y, 
\end{align*}
which is SCSC and has Lipschitz gradients. In \cref{thm:sincoupling}, we characterize its Asplund decomposition and show that the associated acyclic monotone coupling term is non-bilinear. To the best of our knowledge, this provides the first example of a saddle operator whose Asplund decomposition features a non-skew-linear coupling term.

In this work, we introduce \textbf{N}esterov acceleration with \textbf{O}perator \textbf{D}ecomposition (NOD), which extends Nesterov’s accelerated gradient descent (NAG) to monotone inclusion problems. In particular, NOD utilizes operator decomposition into cyclically monotone and monotone components, with the Asplund decomposition serving as the tightest such decomposition.

NOD achieves accelerated convergence rates for strongly monotone-smooth inclusion problems, including SCSC minimax optimization problems. The resulting rates are analogous to recent results under bilinear coupling (BC) \cite{du2022optimal, jin22, kovalev2022accelerated}.

\section{Preliminaries} 
\label{sec:2}
We quickly introduce some standard definitions and notations related to convex functions, monotone operators, and operator decompositions.
Let $\| \cdot \|$ be the standard Euclidean norm for vectors in $\R^{d}$ and the operator norm for matrices in $\R^{d \times d}$.
For functions and operators, we follow the standard notation of the field \cite{lscomo}.

\subsection{Functions}
\label{subsubsec:functions}
We say that a function $f \colon \R^{d} \rightarrow \R\cup\{\infty\}$ is  {convex} if
\begin{align*}
    f (\lambda \vz + (1 - \lambda) \vz') &\le \lambda f (\vz) + (1 - \lambda) f (\vz'), \qquad \forall\, \vz, \vz' \in \R^{d},\, \lambda \in [0, 1].
\end{align*}
We say that $f$ is  {$\mu$-strongly convex} for $\mu > 0$ if $f(\vz) - \frac{\mu}{2} \norm{\vz}^{2}$ is convex.
We say that $f$ is  {($\mu$-strongly) concave} if $-f$ is ($\mu$-strongly) convex.

We say that $\vg \in \R^{d}$ is a  {subgradient} of a convex function $f$ at $\vz$ if
\begin{align*}
    f(\vz') &\ge f(\vz) + \inner{\vg, \vz' - \vz}, \qquad \forall \,\vz' \in \R^{d},
\end{align*}
and the subdifferential $\partial f$ at each point $\vz$ is defined as the collection of subgradients.

A function is  {closed} if the epigraph
$\operatorname{epi} f = \left\{ (x, s) \in \R^{d} \times \mathbb{R} : f(x) \le s\right\}$
is a closed set.
A function is  {proper} if its value is never $-\infty$ and is finite somewhere.
We will often focus on  {convex} functions that are  {closed and proper} and refer to those as {CCP functions}.
It is well known that a proper function is closed if and only if it is lower semicontinuous, and that the subdifferential of a CCP function is always maximal monotone \cite{minty64,moreau65}.
It is also known that if $f$ is CCP, then for all $\vz \in \operatorname{int} \operatorname{dom} f$, $\partial f (\vz)$ is a singleton iff $f$ is differentiable at $z$ (Theorem 25.1, \cite{rockafellarbook}).
Our main results will focus on $L$-smooth CCP functions, which we define below, for which $\operatorname{dom} f = \R^{d}$ and $f$ is differentiable, and thus $\partial f (\vz) = \{ \nabla f(\vz) \}$ is a singleton for all $\vz \in \R^{d}$.

We say that $f\colon \mathbb{R}^d\rightarrow\mathbb{R}$ is  {$L$-smooth} for $L > 0$ if it is differentiable and $\nabla f$ is $L$-Lipschitz, i.e.,
\begin{align*}
    \bignorm{\nabla f(\vz) - \nabla f(\vz')} \le L \bignorm{\vz - \vz'}, \qquad \forall \, \vz, \vz' \in \R^{d}.
\end{align*}
Furthermore, if $f$ is differentiable, the definition of $\mu$-strong convexity becomes equivalent to
\begin{align*}
    f(\vz') &\ge f(\vz) + \inner{\nabla f (\vz), \vz' - \vz} + \frac{\mu}{2} \norm{\vz' - \vz}^{2}, \qquad \forall \, \vz, \vz' \in \R^{d},
\end{align*}
and a convex function $f$ is $L$-smooth if and only if
\begin{align*}
    f(\vz') &\le f(\vz) + \inner{\nabla f (\vz), \vz' - \vz} + \frac{L}{2} \norm{\vz' - \vz}^{2}, \qquad \forall \, \vz, \vz' \in \R^{d}.
\end{align*}

For (unconstrained) minimax optimization problems,  we write $\vz = (\vx, \vy) \in \R^{d_x}\times\R^{d_y}$ for $\vx \in \R^{d_x}$ and $\vy \in \R^{d_y}$.
We say that a function $\mathcal{L} \colon \R^{d_x} \times \R^{d_y} \rightarrow \R$ is $(\mu_{x}, \mu_{y})$-strongly-convex-strongly-concave (SCSC) for $\mu_{x}, \mu_{y} > 0$ if $\mathcal{L} (\cdot, \vy)$ is $\mu_{x}$-strongly convex for all $\vy \in \R^{d_y}$ and $\mathcal{L} (\vx, \cdot)$ is $\mu_{y}$-strongly concave for all $\vx \in \R^{d_x}$.
Also, we say that $\mathcal{L} (\vx, \vy)$ is $(L_{x}, L_{y}, L_{xy})$-smooth if $\nabla_{\vx} \mathcal{L} (\cdot, \vy)$ is $L_{x}$-Lipschitz for all $\vy \in \R^{d_y}$, $\nabla_{\vy} \mathcal{L} (\vx, \cdot)$ is $L_{y}$-Lipschitz for all $\vx \in \R^{d_x}$, and
\begin{align*}
\begin{aligned}    
    \| \nabla_{\vx} \mathcal{L} (\vx, \vy') - \nabla_{\vx} \mathcal{L} (\vx, \vy) \| &\le L_{xy} \| \vy' - \vy \| \\
    \| \nabla_{\vy} \mathcal{L} (\vx', \vy) - \nabla_{\vy} \mathcal{L} (\vx, \vy) \| &\le L_{xy} \| \vx' - \vx \|
    \end{aligned}
    \qquad\quad
\forall\,\vx, \vx' \in \R^{d_x},\,\vy, \vy' \in \R^{d_y}.
\end{align*}
We say $\mathcal{L} $ is $\mu$-SCSC if it is $(\mu, \mu)$-SCSC and $L$-smooth if the saddle operator is $L$-Lipschitz in $\R^{d_x}\times\R^{d_y}$.
An $(L_x, L_y, L_{xy})$-smooth function is $L$-smooth for
\begin{align*}
    L = \left\|
    \begin{bmatrix}
        L_x & L_{xy} \\
        L_{xy} & L_y
    \end{bmatrix}
    \right\|.
\end{align*}

\subsection{Operators}
A set-valued operator $\oT\colon \R^{d} \rightrightarrows \R^{d}$ maps each point $\vz \in \R^{d}$ to a set $\oT(\vz) \subseteq \R^{d}$.
The domain of an operator $\oT$ is defined as $\dom \oT = \{ \vz : \oT(\vz) \ne \emptyset \}$, and $\oT$ is of full domain if $\dom \oT = \R^{d}$.
We say $\oT\colon \R^{d} \rightrightarrows \R^{d}$ is  {monotone} if
\begin{align*}
    \langle \vw - \vw', \vz - \vz' \rangle &\ge 0, \qquad \forall \, 
    \vz,\vz'\in \mathbb{R}^d,\,\vw\in \oT(\vz),\,\vw'\in\oT(\vz').
\end{align*}
We say $\oT\colon \R^{d} \rightrightarrows \R^{d}$ is  {maximal monotone} if its graph has no proper monotone extension.
We say $\oT\colon \R^{d} \rightrightarrows \R^{d}$ is  {$\mu$-strongly monotone} if $\mu > 0$ and
\begin{align*}
    \langle \vw - \vw', \vz - \vz' \rangle &\ge \mu \norm{\vz - \vz'}^{2}, \qquad\forall \, 
    \vz,\vz'\in \mathbb{R}^d,\,\vw\in \oT(\vz),\,\vw'\in\oT(\vz').
\end{align*}
We say $\oT\colon \R^{d} \rightrightarrows \R^{d}$ is $L$-Lipschitz if $L > 0$ and
\begin{align*}
    \norm{\vw - \vw'} &\le L \norm{\vz - \vz'}, \qquad \forall \, 
    \vz,\vz'\in \mathbb{R}^d,\,\vw\in \oT(\vz),\,\vw'\in\oT(\vz').
\end{align*}
In order for a possibly set-valued operator of full domain to be $L$-Lipschitz, it must be single-valued, i.e., $\oT (\vz)$ must be a singleton for all $\vz\in \mathbb{R}^d$.
We write $\oT\colon \R^{d} \rightarrow \R^{d}$ and $\oT(\vz) \in \R^{d}$ for single-valued operators with a slight abuse of notation.
If an operator $\oS$ is Lipschitz, we will sometimes write $L_{\oS}$ to denote its Lipschitz constant. 

If $\oT$ is maximal strongly monotone, then it must have a unique zero $\vz_{\star}$ by the Minty surjectivity theorem \cite{browder68,minty62}.
This includes the case when $\oT$ is the saddle gradient of a SCSC function defined on $\R^{d_x}\times\R^{d_y}$, which is single-valued, continuous, and of full domain, which implies that $\oT$ is maximal \cite{browder68}.
Our main results will focus on maximal strongly monotone, Lipschitz operators $\oT$ of full domain.

\subsection{Operator decompositions}
\label{subsubsec:decomp}

Now we formally define \textit{acyclic monotone} operators and the Asplund decomposition of monotone operators.
\begin{definition}
    We say that a monotone operator $\oS \colon \mathbb{R}^d \rightrightarrows \mathbb{R}^d$ is \emph{acyclic monotone} if whenever $\oS = \partial k + \oA$ for some CCP $k \colon \mathbb{R}^d \rightarrow \mathbb{R} \cup \{\infty\}$ and monotone $\oA \colon \mathbb{R}^d \rightrightarrows \mathbb{R}^d$, then $k$ is affine on $\operatorname{dom} \oS$.
\end{definition}
The $+$ operation between operators is defined as Minkowski sums, i.e., the operator $\oA + \oB \colon \R^{d} \rightrightarrows \R^{d}$ such that $(\oA + \oB) (\vz) = \{ \va + \vb : \va \in \oA (\vz), \, \vb \in \oB (\vz) \}$.
\begin{definition}
    For a maximal monotone $\oT \colon \mathbb{R}^d \rightrightarrows \mathbb{R}^d$, an \emph{Asplund decomposition} is a decomposition of the form $\oT = \partial \phi + \oS$ such that $\phi \colon \mathbb{R}^d \rightarrow \mathbb{R} \cup \{\infty\}$ is CCP and $\oS \colon \mathbb{R}^d \rightrightarrows \mathbb{R}^d$ is acyclic monotone.
\end{definition}
\begin{theorem}[Theorem 1, \cite{asplund70}]
    \label{thm:asplund}
    For any single-valued, maximal monotone operator $\oT \colon \mathbb{R}^d \rightarrow \mathbb{R}^d$ with $\operatorname{int} \operatorname{dom} \oT \ne \emptyset$, an \emph{Asplund decomposition} exists.
\end{theorem}

If $\oT \colon \mathbb{R}^d \to \mathbb{R}^d$ is maximal monotone and $L$-Lipschitz for some $L < \infty$ and we have the decomposition (not necessarily Asplund)
\[
\oT = \partial \phi + \oS
\]
with CCP $\phi \colon \mathbb{R}^d \rightarrow \mathbb{R} \cup \{\infty\}$ and maximal monotone $\oS \colon \mathbb{R}^d \rightrightarrows \mathbb{R}^d$, we can show that $\phi$ is, in fact, $L$-smooth and thus differentiable (as we show in \Cref{lem:kissmooth} in the appendix), and $\oS$ is $2L$-Lipschitz as it is the difference of two $L$-Lipschitz operators. 
If we further assume that $\oT$ is $\mu$-strongly monotone, then there is an Asplund decomposition $\oT = \partial \phi + \oS$ such that $\phi$ is $\mu$-strongly convex, obtained by finding an Asplund decomposition of $\oT - \mu \cdot \mathrm{id}$, where $\mathrm{id}$ is the identity operator.

We note that the proof of \Cref{thm:asplund} is based on the application of Zorn's lemma, and thus is non-constructive.
To the best of our knowledge, there is no known nontrivial class of monotone or saddle operators for which the Asplund decomposition can be written explicitly, beyond saddle operators induced by convex--concave functions \textit{with bilinear coupling}.

Interestingly, we identify a saddle operator outside this class whose Asplund decomposition admits a closed-form expression.
\begin{theorem}
    \label{thm:sincoupling}
    Let $\mathcal{L} \colon \R^2 \rightarrow \R$ be
    \begin{align*}
        \mathcal{L} (x, y) &= x^2 - y^2 + \sin x \sin y.
    \end{align*}
    Then, the saddle gradient of $\gL$,
    \begin{align*}
        \oT (x, y) = \begin{bmatrix}
            \partial_x \mathcal{L} (x, y) \\
            - \partial_y \mathcal{L} (x, y)
        \end{bmatrix}
        &= 
        \begin{bmatrix}
            2x + \cos x \sin y \\
            2y - \sin x \cos y
        \end{bmatrix},
    \end{align*}
    has an Asplund decomposition of the form $\oT = \nabla \phi + \oS$, where
    \begin{align*}
        \phi(x, y) &:= h(x) + h(y), \qquad \qquad h(w) := \int_0^w \int_0^t (2 - \lvert \sin s \rvert) \, \dd s \, \dd t, \\
        \oS (x, y) &:= \oT (x, y) - \nabla \phi(x, y).
    \end{align*}
\end{theorem}

We defer the proof of \Cref{thm:sincoupling} to \Cref{app:sincoupling}.
The proof does not utilize any high-level operator-theoretic abstractions, but is technically involved. We believe there is room to further develop operator structure theory that would simplify such questions, but we leave this direction for future work.

\begin{example}[Saddle gradients]
    \label{ex:2}
    Suppose that $\vz = (\vx, \vy) \in \R^{d_x}\times\R^{d_y}$ for $\vx \in \R^{d_x}$ and $\vy \in \R^{d_y}$, and $\mathcal{L}$ is an SCSC, smooth function with \emph{bilinear coupling}:
    \begin{align*}
        \mathcal{L} (\vx,\vy) = g(\vx) - h(\vy) + \vx^\top \mM \vy.
    \end{align*}
    Then we can decompose the saddle gradient of $\mathcal{L}$, defined by
    \begin{align*}
        \oT(\vx, \vy) = \begin{bmatrix}
            \nabla_{\vx} \mathcal{L} (\vx, \vy) \\
            - \nabla_{\vy} \mathcal{L} (\vx, \vy)
        \end{bmatrix},
    \end{align*}
    into $\oT = \nabla \phi + \oS$ with the following components,
    \begin{align*}
        \begin{aligned}
        \phi (\vx, \vy) &= g(\vx) + h(\vy), \\
        \oS (\vx, \vy) &=  \begin{bmatrix}
            \bm{0} & \mM \\
            -\mM^{\top} & \bm{0}
        \end{bmatrix} 
        \begin{bmatrix}
            \vx \\
            \vy
        \end{bmatrix}.
        \end{aligned}
    \end{align*}
    We can observe that $\phi$ is strongly convex and $\oS$ is monotone.
    In particular, $\oS$ is a
    \textit{skew linear operator,} which is known to be acyclic monotone \cite{asplund}.
    Furthermore, if $g$ is $\mu_x$-strongly convex and $L_x$-smooth, $h$ is $\mu_y$-strongly convex and $L_y$-smooth, and $\mM$ has bounded spectral norm $\| \mM \| \le L_{xy}$, then $\mathcal{L}$ is $(\mu_x, \mu_y)$-SCSC and $(L_x, L_y, L_{xy})$-smooth.
\end{example}

\section{ODE formulation}
\label{sec:3}

Before we propose and analyze the (discrete-time) algorithm NOD, we present a motivating continuous-time formulation:
\begin{align}
    \Ddot{\rvz}_{t} + 2 \sqrt{\mu} \dot{\rvz}_{t} + \nabla \phi (\rvz_{t}) + \oS \bigopen{\rvz_{t} + \frac{1}{\sqrt{\mu}} \dot{\rvz}_{t}} = \bm{0}
    \label{eq:ode}
\end{align}
for $t>0$.
(We use the notation $\rvz_{t}$ for continuous-time flows and $\vz_{k}$ for discrete-time algorithms.)
If $\oS = 0$, then \eqref{eq:ode} recovers the continuous-time ODE of NAG \cite{su16} for a strongly convex function $\phi$.
We show that the ODE in \eqref{eq:ode} enjoys exponential stability.

We assume that there exists a unique $\vz_{\star}$ such that
\[
\nabla \phi(\vz_{\star}) + \oS (\vz_{\star}) = \bm{0},
\]
which holds when $\oT = \nabla \phi + \oS$ is a maximal strongly monotone operator.

\begin{restatable}{theorem}{thmode}
    \label{thm:ode}
    Let $\phi\colon \mathbb{R}^d\rightarrow\mathbb{R}$ be a differentiable and $\mu$-strongly convex function and $\oS\colon \R^{d} \rightarrow \R^{d}$ be a (single-valued and) monotone operator.
    Let $\{\rvz_{t}\}_{t\ge 0}$ be a solution of the ODE \eqref{eq:ode} and $\vz_{\star}$ be the (unique) point satisfying $\nabla \phi(\vz_{\star}) + \oS (\vz_{\star}) = \bm{0}$.
    Then, the Lyapunov function 
    \begin{align*}
        \Psi_{t} &= \bignorm{\dot{\rvz}_{t} + \sqrt{\mu} (\rvz_{t} - \vz_{\star})}^2 + 2 \big( \phi (\rvz_{t}) - \phi (\vz_{\star}) - \langle \nabla \phi (\vz_{\star}),\rvz_{t}-\vz_{\star} \rangle \big)
    \end{align*}
    satisfies $\dot{\Psi}_{t} \le - \sqrt{\mu} \Psi_{t}$.
\end{restatable}

\begin{proof}
Observe that
\begin{align*}
    \frac{\dd}{\dd t} \bignorm{\dot{\rvz}_{t} + \sqrt{\mu} (\rvz_{t} - \vz_{\star})}^2 &= 2 \inner{\Ddot{\rvz}_{t} + \sqrt{\mu} \dot{\rvz}_{t}, \dot{\rvz}_{t} + \sqrt{\mu} (\rvz_{t} - \vz_{\star})} \\
    &{}={} - 2 \inner{\sqrt{\mu} \dot{\rvz}_{t} + \nabla \phi (\rvz_{t}), \dot{\rvz}_{t} + \sqrt{\mu} (\rvz_{t} - \vz_{\star})} \\
    &\phantom{{}={}} - 2 \inner{\oS \bigopen{\rvz_{t} + \frac{1}{\sqrt{\mu}} \dot{\rvz}_{t}}, \dot{\rvz}_{t} + \sqrt{\mu} (\rvz_{t} - \vz_{\star})}.
\end{align*}
First, we have
\begin{align*}
    &- 2 \inner{\sqrt{\mu} \dot{\rvz}_{t}, \dot{\rvz}_{t} + \sqrt{\mu} (\rvz_{t} - \vz_{\star})} = - \sqrt{\mu} \bigopen{\bignorm{\dot{\rvz}_{t} + \sqrt{\mu} (\rvz_{t} - \vz_{\star})}^2 + \bignorm{\dot{\rvz}_{t}}^2 - \mu \bignorm{\rvz_{t} - \vz_{\star}}^2}
\end{align*}
and by monotonicity of $\oS$, we have
\begin{align*}
    \inner{\oS \bigopen{\rvz_{t} + \frac{1}{\sqrt{\mu}} \dot{\rvz}_{t}}, \dot{\rvz}_{t} + \sqrt{\mu} (\rvz_{t} - \vz_{\star})} &= \sqrt{\mu} \inner{\oS \bigopen{\rvz_{t} + \frac{1}{\sqrt{\mu}} \dot{\rvz}_{t}}, \bigopen{\rvz_{t} + \frac{1}{\sqrt{\mu}} \dot{\rvz}_{t}} - \vz_{\star}} \\
    &\ge \sqrt{\mu} \inner{\oS (\vz_{\star}), \bigopen{\rvz_{t} + \frac{1}{\sqrt{\mu}} \dot{\rvz}_{t}} - \vz_{\star}} \\
    &= \sqrt{\mu} \inner{\oS (\vz_{\star}), \rvz_{t} - \vz_{\star}} + \inner{\oS (\vz_{\star}), \dot{\rvz}_{t}} \\
    &= - \sqrt{\mu} \inner{\nabla \phi (\vz_{\star}), \rvz_{t} - \vz_{\star}} - \inner{\nabla \phi (\vz_{\star}), \dot{\rvz}_{t}}
\end{align*}
and therefore
\begin{align*}
    - 2 \inner{\oS \bigopen{\rvz_{t} + \frac{1}{\sqrt{\mu}} \dot{\rvz}_{t}}, \dot{\rvz}_{t} + \sqrt{\mu} (\rvz_{t} - \vz_{\star})} &\le 2 \sqrt{\mu} \inner{\nabla \phi (\vz_{\star}), \rvz_{t} - \vz_{\star}} + 2 \inner{\nabla \phi (\vz_{\star}), \dot{\rvz}_{t}}.
\end{align*}
Also, by the chain rule we have
\begin{align*}
    \frac{\dd}{\dd t} 2 (\phi (\rvz_{t}) - \phi (\vz_{\star}) - \langle \nabla \phi (\vz_{\star}), \rvz_{t}-\vz_{\star} \rangle) &= 2 \inner{\nabla \phi (\rvz_{t}) - \nabla \phi (\vz_{\star}), \dot{\rvz}_{t}},
\end{align*}
and by strong convexity of $\phi$ we have
\begin{align*}
    - 2 \inner{\nabla \phi (\rvz_{t}), \sqrt{\mu} (\rvz_{t} - \vz_{\star})} &\le - \sqrt{\mu} \bigopen{2 \bigopen{\phi (\rvz_{t}) - \phi (\vz_{\star})} + \mu \bignorm{\rvz_{t} - \vz_{\star}}^2}.
\end{align*}
Therefore we can conclude that
\begin{align*}
    \dot{\Psi}_{t} &= \frac{\dd}{\dd t} \bigopen{\bignorm{\dot{\rvz}_{t} + \sqrt{\mu} (\rvz_{t} - \vz_{\star})}^2 + 2 \big( \phi (\rvz_{t}) - \phi (\vz_{\star}) - \langle \nabla \phi (\vz_{\star}), \rvz_{t}-\vz_{\star} \rangle \big)} \\
    &\le -\sqrt{\mu} \bigopen{\bignorm{\dot{\rvz}_{t} + \sqrt{\mu} (\rvz_{t} - \vz_{\star})}^2 + 2 \big( \phi (\rvz_{t}) - \phi (\vz_{\star}) - \langle \nabla \phi (\vz_{\star}), \rvz_{t}-\vz_{\star} \rangle \big) + \bignorm{\dot{\rvz}_{t}}^2} \\
    &\le - \sqrt{\mu} \Psi_{t},
\end{align*}
and thus $\dot{\Psi}_{t} \le - \sqrt{\mu} \Psi_{t}$ as desired.
\end{proof}

Since
$\phi$ is $\mu$-strongly convex,
\begin{align*}
    \Psi_{t} &\ge 2 \bigopen{\phi (\rvz_{t}) - \phi (\vz_{\star}) - \langle \nabla \phi (\vz_{\star}), \rvz_t - \vz_{\star} \rangle} \ge \mu \bignorm{\rvz_{t} - \vz_{\star}}^2,
\end{align*}
and by Gr\"onwall's lemma, 
\begin{align*}
    \mu \bignorm{\rvz_{t} - \vz_{\star}}^2 \le \Psi_{t} \le \Psi_{0} \cdot e^{- \sqrt{\mu} t}.
\end{align*}
Thus $\rvz_{t} \rightarrow  \vz_{\star}$ at a rate analogous to that of the continuous-time ODE of NAG \cite{su16,wilson21}.

Note that the decomposition $\oT = \nabla \phi + \oS$ need not be Asplund.
Nevertheless, the Asplund decomposition can be interpreted as one that maximally allocates curvature to $\phi$, and the resulting curvature parameter $\mu$ of $\phi$ governs the rate of convergence.

\section{Algorithm}
\label{sec:4}

We now present a discretization of \eqref{eq:ode}, which we call \textbf{N}esterov acceleration with \textbf{O}perator \textbf{D}ecomposition (NOD):
\begin{align}
    \begin{aligned}
    \vz_{k+1} &= \vztilde_{k} - \eta \bigopen{\nabla \phi (\vztilde_{k}) + \oS (\vzhat_{k})}\\
    \vztilde_{k+1} &= \vz_{k+1} + \tau (\vz_{k+1} - \vz_{k})\\
    \vzhat_{k+1} &= \vztilde_{k+1} + \theta (\vztilde_{k+1} - \vztilde_{k})
    \end{aligned}
    \label{eq:nod}
\end{align}
for $k=0,1,\dots$, where $\vz_{0} = \vztilde_{0} = \vzhat_{0}$ is the starting point
and $\tau = \frac{1 - \sqrt{\eta \mu}}{1 + \sqrt{\eta \mu}}$, $ \theta = \frac{1}{\sqrt{\eta \mu}} - 1$. We use a modified Nesterov-type momentum to compute extrapolated iterates $\vztilde$ and $\vzhat$ using momentum parameters $\tau$ and $ \theta$.
The $\vztilde$-momentum is identical to the traditional Nesterov momentum for strongly convex functions, while the $\vzhat$-momentum takes a larger extrapolation step.
In the special case of $\oS \equiv \bm{0}$, we can eliminate $\vzhat_{k}$ and NOD reduces to the classical NAG on strongly convex functions \cite{nesterov04}.
Finally, the continuous-time limit of NOD recovers the earlier ODE \eqref{eq:ode} under the usual accelerated time scaling $t \propto k \sqrt{\eta}$.

Now we present the main theorem, \cref{thm:discretize}, which proves linear convergence of NOD with an accelerated rate analogous to that of NAG.

\begin{restatable}{theorem}{thmdiscretize}
    \label{thm:discretize}
    Suppose that $\phi\colon \R^{d} \rightarrow \R$ is a $\mu$-strongly convex, $L_{\phi}$-smooth function and $\oS\colon \R^{d} \rightarrow \R^{d}$ is a monotone, $L_{\oS}$-Lipschitz operator.
    Consider the \emph{NOD} iterates as in \eqref{eq:nod}, and let $\vz_{\star}$ be the (unique) point satisfying $\nabla \phi(\vz_{\star}) + \oS (\vz_{\star}) = \bm{0}$.
    For $k \ge 2$, let us define
    \begin{align*}
        \begin{aligned}
        \Psi_{k} &= \mu \bignorm{\vztilde_{k} + \frac{1}{\sqrt{\eta \mu}} \bigopen{\vztilde_{k} - \vz_{k}} - \vz_{\star}}^2 + 2 \hat{\phi}(\vztilde_{k-1}) \\
        &\phantom{{}={}} - \eta \bignorm{\nabla \phi (\vztilde_{k-1}) + \oS (\vzhat_{k-1})}^2 + (1 - \eta \mu) \sqrt{\frac{\mu}{\eta}} \bignorm{\vztilde_{k-1} - \vz_{k-1}}^{2} \\
        &\phantom{{}={}} + B \eta (1 - \sqrt{\eta \mu}) (1 - \eta L_{\phi}) \bignorm{\nabla \phi (\vztilde_{k-2}) + \oS (\vzhat_{k-2})}^2,
        \end{aligned}
    \end{align*}
    where
    \begin{align}
        \hat{\phi} (\vz) &= \phi(\vz) - \phi(\vz_{\star}) - \langle \nabla \phi (\vz_{\star}), \vz - \vz_{\star} \rangle,
        \label{eq:affineshift}
    \end{align}
    $B = \frac{\sqrt{C}}{1-C}$,
    and $C (\approx 0.026118)$ is the unique real root of
    $1 - C - \sqrt{C} \bigopen{C + 6} = 0$.
    Then, for any choice of a positive step size with\footnote{If $L_{\oS} = 0$, then we can substitute $L_{\oS}$ with any small positive value and choose $\eta = \frac{C}{L_{\phi}}$.}
    \begin{align*}
        \eta \le C \cdot \min \bigset{\frac{1}{L_{\phi}}, \frac{\mu}{L_{\oS}^{2}}}, 
    \end{align*}
    we have $\Psi_{k+1} \le (1 - \sqrt{\eta \mu}) \Psi_{k}$ for $k \ge 2$.    
\end{restatable}

\begin{proof}
Define
\begin{align}
    \hat{\oS}(\vz) &= \oS(\vz) - \oS (\vz_{\star}) = \oS(\vz) + \nabla \phi (\vz_{\star}).
    \label{eq:affineshift2}
\end{align}
Then $\nabla \phi + \oS = \nabla \hat{\phi} + \hat{\oS}$, $\hat{\phi} (\vz_{\star}) = 0$, and $\nabla \hat{\phi} (\vz_{\star}) = \hat{\oS}(\vz_{\star}) = \mathbf{0}$.
Therefore we have
\begin{align*}
    \begin{aligned}
    \Psi_{k} &= \mu \bignorm{\vztilde_{k} + \frac{1}{\sqrt{\eta \mu}} \bigopen{\vztilde_{k} - \vz_{k}} - \vz_{\star}}^2 + 2 \hat{\phi}(\vztilde_{k-1}) \\
    &\phantom{{}={}} - \eta \| \nabla \hat{\phi} (\vztilde_{k-1}) + \hat{\oS} (\vzhat_{k-1}) \|^2 + (1 - \eta \mu) \sqrt{\frac{\mu}{\eta}} \bignorm{\vztilde_{k-1} - \vz_{k-1}}^{2} \\
    &\phantom{{}={}} + B \eta (1 - \sqrt{\eta \mu}) (1 - \eta L_{\phi}) \| \nabla \hat{\phi} (\vztilde_{k-2}) + \hat{\oS} (\vzhat_{k-2}) \|^2.
    \end{aligned} 
\end{align*}
The transformations in \eqref{eq:affineshift} and \eqref{eq:affineshift2} preserve the curvature modulus of $\phi$ (strong convexity and smoothness) and properties of $\oS$ (monotonicity and Lipschitzness).

Observe that the algorithm iterates satisfy
\begin{align*}
    \vztilde_{k+1} - \vz_{k} + \frac{1}{\sqrt{\eta \mu}} (\vztilde_{k+1} - \vz_{k+1})
    &= \bigopen{\frac{1 - \sqrt{\eta \mu}}{1 + \sqrt{\eta \mu}} + 1 + \frac{1}{\sqrt{\eta \mu}} \cdot \frac{1 - \sqrt{\eta \mu}}{1 + \sqrt{\eta \mu}}} (\vz_{k+1} - \vz_{k}) \\
    &= \frac{1}{\sqrt{\eta \mu}} (\vz_{k+1} - \vz_{k}).
\end{align*}
Therefore we can write
\begin{align}
    \mu \bignorm{\vztilde_{k+1} - \vz_{\star} + \frac{1}{\sqrt{\eta \mu}} (\vztilde_{k+1} - \vz_{k+1})}^2 &= \mu \bignorm{\vz_{k} - \vz_{\star} + \frac{1}{\sqrt{\eta \mu}} (\vz_{k+1} - \vz_{k})}^2. \label{eq:nag1}
\end{align}
Plugging in $\vz_{k+1} = \vztilde_{k} - \eta (\nabla \phi (\vztilde_{k}) + \oS (\vzhat_{k})) = \vztilde_{k} - \eta (\nabla \hat{\phi} (\vztilde_{k}) + \hat{\oS} (\vzhat_{k}))$ and rearranging terms, \eqref{eq:nag1} is equivalent to
\begin{align}
    &\mu \bignorm{\vz_{k} - \vz_{\star} + \frac{1}{\sqrt{\eta \mu}} (\vztilde_{k} - \eta \bigopen{\nabla \hat{\phi} (\vztilde_{k}) + \hat{\oS} (\vzhat_{k})} - \vz_{k})}^2 \notag \\
    &\begin{aligned}
    &= \mu \bignorm{\vztilde_{k} - \vz_{\star} - \sqrt{\frac{\eta}{\mu}} \nabla \hat{\phi} (\vztilde_{k}) + \bigopen{\frac{1}{\sqrt{\eta \mu}} - 1} (\vztilde_{k} - \vz_{k})}^2 + \eta \| \hat{\oS} (\vzhat_{k}) \|^2 \\
    &\phantom{{}={}} - 2 \sqrt{\eta \mu} \inner{\hat{\oS} (\vzhat_{k}), \vztilde_{k} - \vz_{\star} - \sqrt{\frac{\eta}{\mu}} \nabla \hat{\phi} (\vztilde_{k}) + \bigopen{\frac{1}{\sqrt{\eta \mu}} - 1} (\vztilde_{k} - \vz_{k})}.
    \end{aligned}
    \label{eq:nag2}
\end{align}
We can expand
\begin{align}
    &\mu \bignorm{\vztilde_{k} - \vz_{\star} - \sqrt{\frac{\eta}{\mu}} \nabla \hat{\phi} (\vztilde_{k}) + \bigopen{\frac{1}{\sqrt{\eta \mu}} - 1} (\vztilde_{k} - \vz_{k})}^2 \notag \\
    &\begin{aligned}
    &= \mu \bignorm{\vztilde_{k} - \vz_{\star}}^2 + \eta \| \nabla \hat{\phi} (\vztilde_{k}) \|^2 + 
    \bigopen{1 - \sqrt{\eta \mu}}^2 \cdot \frac{1}{\eta} \bignorm{\vztilde_{k} - \vz_{k}}^2 \\
    &\phantom{{}={}} - 2 \sqrt{\eta \mu} \langle \vztilde_{k} - \vz_{\star}, \nabla \hat{\phi} (\vztilde_{k}) \rangle - 2 \bigopen{1 - \sqrt{\eta \mu}} \langle \vztilde_{k} - \vz_{k}, \nabla \hat{\phi} (\vztilde_{k}) \rangle \\
    &\phantom{{}={}} + 2 \bigopen{1 - \sqrt{\eta \mu}} \sqrt{\frac{\mu}{\eta}} \inner{\vztilde_{k} - \vz_{k}, \vztilde_{k} - \vz_{\star}}.
    \end{aligned} \label{eq:cvxmidresult}
\end{align}
By strong convexity and smoothness of $\hat{\phi}$ and $\hat{\phi} (\vz_{\star}) = 0$, we have
\begin{align}
    - 2 \langle \vztilde_{k} - \vz_{\star}, \nabla \hat{\phi} (\vztilde_{k}) \rangle &\le - 2 \hat{\phi} (\vztilde_{k}) - \mu \bignorm{\vztilde_{k} - \vz_{\star}}^2, \label{eq:sc1} \\
    - 2 \langle \vztilde_{k} - \vz_{k}, \nabla \hat{\phi} (\vztilde_{k}) \rangle &\le 2 (\hat{\phi} (\vz_{k}) - \hat{\phi} (\vztilde_{k})) - \mu \bignorm{\vztilde_{k} - \vz_{k}}^2, \label{eq:sc2} \\
    2 ( \hat{\phi} (\vz_{k}) - \hat{\phi} (\vztilde_{k-1}) ) &\le 2 \langle \vz_{k} - \vztilde_{k-1}, \nabla \hat{\phi} (\vztilde_{k-1}) \rangle + L_{\phi} \bignorm{\vz_{k} - \vztilde_{k-1}}^2. \label{eq:smooth1}
\end{align}
Thus, $\sqrt{\eta \mu} \times \eqref{eq:sc1} + (1 - \sqrt{\eta \mu}) (\eqref{eq:sc2} + \eqref{eq:smooth1})$ yields
\begin{align*}
    &- 2\sqrt{\eta \mu} \langle \vztilde_{k} - \vz_{\star}, \nabla \hat{\phi} (\vztilde_{k}) \rangle - 2 \bigopen{1 - \sqrt{\eta \mu}} \langle \vztilde_{k} - \vz_{k}, \nabla \hat{\phi} (\vztilde_{k}) \rangle \\
    &\le - \sqrt{\eta \mu} \cdot \mu \norm{\vztilde_{k} - \vz_{\star}}^2 - 2 \sqrt{\eta \mu} \, \hat{\phi} (\vztilde_{k}) + 2 \bigopen{1 - \sqrt{\eta \mu}} ( \hat{\phi} (\vztilde_{k-1}) - \hat{\phi} (\vztilde_{k}) ) \\
    &\phantom{\le{}} - \bigopen{1 - \sqrt{\eta \mu}} \mu \norm{\vztilde_{k} - \vz_{k}}^2 + 2 \bigopen{1 - \sqrt{\eta \mu}} \langle \vz_{k} - \vztilde_{k-1}, \nabla \hat{\phi} (\vztilde_{k-1}) \rangle \\
    &\phantom{\le{}} + \bigopen{1 - \sqrt{\eta \mu}} L_{\phi} \bignorm{\vz_{k} - \vztilde_{k-1}}^2 \\
    &= - \sqrt{\eta \mu} \cdot \mu \norm{\vztilde_{k} - \vz_{\star}}^2 - 2 \hat{\phi} (\vztilde_{k}) + 2 \bigopen{1 - \sqrt{\eta \mu}} \hat{\phi} (\vztilde_{k-1}) \\
    &\phantom{\le{}} - \bigopen{1 - \sqrt{\eta \mu}} \mu \norm{\vztilde_{k} - \vz_{k}}^2 + 2 \bigopen{1 - \sqrt{\eta \mu}} \langle \vz_{k} - \vztilde_{k-1}, \nabla \hat{\phi} (\vztilde_{k-1}) \rangle \\
    &\phantom{\le{}} + \bigopen{1 - \sqrt{\eta \mu}} L_{\phi} \bignorm{\vz_{k} - \vztilde_{k-1}}^2
\end{align*}
and therefore we obtain an upper bound of \eqref{eq:cvxmidresult},
\begin{align}
    &\begin{aligned}
    &\mu \bignorm{\vztilde_{k} - \vz_{\star}}^2 + \eta \| \nabla \hat{\phi} (\vztilde_{k}) \|^2 + 
    \bigopen{1 - \sqrt{\eta \mu}}^2 \cdot \frac{1}{\eta} \bignorm{\vztilde_{k} - \vz_{k}}^2 \\
    &- \sqrt{\eta \mu} \cdot \mu \norm{\vztilde_{k} - \vz_{\star}}^2 - 2 \hat{\phi} (\vztilde_{k}) + 2 \bigopen{1 - \sqrt{\eta \mu}} \hat{\phi} (\vztilde_{k-1}) \\
    &- \bigopen{1 - \sqrt{\eta \mu}} \mu \norm{\vztilde_{k} - \vz_{k}}^2 + 2 \bigopen{1 - \sqrt{\eta \mu}} \langle \vz_{k} - \vztilde_{k-1}, \nabla \hat{\phi} (\vztilde_{k-1}) \rangle \\
    & + \bigopen{1 - \sqrt{\eta \mu}} L_{\phi} \bignorm{\vz_{k} - \vztilde_{k-1}}^2 + 2 \bigopen{1 - \sqrt{\eta \mu}} \sqrt{\frac{\mu}{\eta}} \inner{\vztilde_{k} - \vz_{k}, \vztilde_{k} - \vz_{\star}}
    \end{aligned} \notag \\
    &\begin{aligned}
    &= \mu \bigopen{1 - \sqrt{\eta \mu}} \bignorm{\vztilde_{k} - \vz_{\star}}^2 + 2 \bigopen{1 - \sqrt{\eta \mu}} \sqrt{\frac{\mu}{\eta}} \inner{\vztilde_{k} - \vz_{k}, \vztilde_{k} - \vz_{\star}} \\
    &\phantom{={}} + \bigopen{1 - \sqrt{\eta \mu}} \frac{1}{\eta} \bignorm{\vztilde_{k} - \vz_{k}}^2 - \bigopen{1 - \eta \mu} \sqrt{\frac{\mu}{\eta}} \bignorm{\vztilde_{k} - \vz_{k}}^2- 2 \hat{\phi} (\vztilde_{k}) \\
    &\phantom{={}} + 2 \bigopen{1 - \sqrt{\eta \mu}} \hat{\phi} (\vztilde_{k-1}) + 2 \bigopen{1 - \sqrt{\eta \mu}} \langle \vz_{k} - \vztilde_{k-1}, \nabla \hat{\phi} (\vztilde_{k-1}) \rangle \\
    &\phantom{={}} + \bigopen{1 - \sqrt{\eta \mu}} L_{\phi} \bignorm{\vz_{k} - \vztilde_{k-1}}^2 + \eta \| \nabla \hat{\phi} (\vztilde_{k}) \|^2
    \end{aligned} \label{eq:trent} \\
    &\begin{aligned}
    &= \mu \bigopen{1 - \sqrt{\eta \mu}} \bignorm{\vztilde_{k} - \vz_{\star} + \frac{1}{\sqrt{\eta \mu}} (\vztilde_{k} - \vz_{k})}^2  - 2 \hat{\phi} (\vztilde_{k}) + 2 \bigopen{1 - \sqrt{\eta \mu}} \hat{\phi} (\vztilde_{k-1}) \\
    &\phantom{={}} + 2 \bigopen{1 - \sqrt{\eta \mu}} \langle \vz_{k} - \vztilde_{k-1}, \nabla \hat{\phi} (\vztilde_{k-1}) \rangle + \bigopen{1 - \sqrt{\eta \mu}} L_{\phi} \bignorm{\vz_{k} - \vztilde_{k-1}}^2 \\
    &\phantom{={}} - \bigopen{1 - \eta \mu} \sqrt{\frac{\mu}{\eta}} \bignorm{\vztilde_{k} - \vz_{k}}^2 + \eta \| \nabla \hat{\phi} (\vztilde_{k}) \|^2.
    \end{aligned}
    \label{eq:alexander}
\end{align}
Note that in \eqref{eq:trent} we use
\begin{align*}
    &\bigopen{1 - \sqrt{\eta \mu}}^2 \frac{1}{\eta} \bignorm{\vztilde_{k} - \vz_{k}}^2 - \bigopen{1 - \sqrt{\eta \mu}} \mu \norm{\vztilde_{k} - \vz_{k}}^2 \\
    &= \bigopen{1 - \sqrt{\eta \mu}} \frac{1}{\eta} \bignorm{\vztilde_{k} - \vz_{k}}^2 - \bigopen{1 - \eta \mu} \sqrt{\frac{\mu}{\eta}} \bignorm{\vztilde_{k} - \vz_{k}}^2.
\end{align*}
Plugging in $\vz_{k} = \vztilde_{k-1} - \eta ( \nabla \hat{\phi} (\vztilde_{k-1}) + \hat{\oS} (\vzhat_{k-1}) )$, we have
\begin{align}
    \begin{aligned}
    &2 \bigopen{1 - \sqrt{\eta \mu}} \langle \vz_{k} - \vztilde_{k-1}, \nabla \hat{\phi} (\vztilde_{k-1}) \rangle + \bigopen{1 - \sqrt{\eta \mu}} L_{\phi} \bignorm{\vz_{k} - \vztilde_{k-1}}^2 \\
    &= - 2 \eta \bigopen{1 - \sqrt{\eta \mu}} \langle \nabla \hat{\phi} (\vztilde_{k-1}), \nabla \hat{\phi} (\vztilde_{k-1}) + \hat{\oS} (\vzhat_{k-1}) \rangle \\
    &\phantom{={}} + \eta^{2} L_{\phi} \bigopen{1 - \sqrt{\eta \mu}} \| \nabla \hat{\phi} (\vztilde_{k-1}) + \hat{\oS} (\vzhat_{k-1}) \|^{2}.
    \end{aligned} \label{eq:arnold}
\end{align}
Therefore, from \eqref{eq:nag2}, \eqref{eq:alexander}, and \eqref{eq:arnold},
\begin{align}
    &\begin{aligned}
    &\mu \bignorm{\vztilde_{k+1} - \vz_{\star} + \frac{1}{\sqrt{\eta \mu}} (\vztilde_{k+1} - \vz_{k+1})}^2 + 2 \hat{\phi} (\vztilde_{k}) \\
    &\le \bigopen{1 - \sqrt{\eta \mu}} \bigopen{\mu \bignorm{\vztilde_{k} - \vz_{\star} + \frac{1}{\sqrt{\eta \mu}} (\vztilde_{k} - \vz_{k})}^2 + 2 \hat{\phi} (\vztilde_{k-1})} \\
    &\phantom{\le{}} - 2 \eta \bigopen{1 - \sqrt{\eta \mu}} \langle \nabla \hat{\phi} (\vztilde_{k-1}), \nabla \hat{\phi} (\vztilde_{k-1}) + \hat{\oS} (\vzhat_{k-1}) \rangle \\
    &\phantom{\le{}} + \eta^{2} L_{\phi} \bigopen{1 - \sqrt{\eta \mu}} \| \nabla \hat{\phi} (\vztilde_{k-1}) + \hat{\oS} (\vzhat_{k-1}) \|^{2} \\
    &\phantom{\le{}} - \bigopen{1 - \eta \mu} \sqrt{\frac{\mu}{\eta}} \bignorm{\vztilde_{k} - \vz_{k}}^2 + \eta \| \nabla \hat{\phi} (\vztilde_{k}) \|^2 + \eta \| \hat{\oS} (\vzhat_{k}) \|^2 \\
    &\phantom{\le{}} - 2 \sqrt{\eta \mu} \inner{\hat{\oS} (\vzhat_{k}), \vztilde_{k} - \vz_{\star} - \sqrt{\frac{\eta}{\mu}} \nabla \hat{\phi} (\vztilde_{k}) + \bigopen{\frac{1}{\sqrt{\eta \mu}} - 1} (\vztilde_{k} - \vz_{k})}.
    \end{aligned} \label{eq:endofstepone}
\end{align}

Recall that
\begin{align*}
    \vzhat_{k} &= \vztilde_{k} + \bigopen{\frac{1}{\sqrt{\eta \mu}} - 1} (\vztilde_{k} - \vztilde_{k-1}).
\end{align*}
In \eqref{eq:endofstepone}, observe that the last inner product term can be rewritten as
\begin{align*}
    & - 2 \sqrt{\eta \mu} \inner{\hat{\oS} (\vzhat_{k}), \vztilde_{k} - \vz_{\star} - \sqrt{\frac{\eta}{\mu}} \nabla \hat{\phi} (\vztilde_{k}) + \bigopen{\frac{1}{\sqrt{\eta \mu}} - 1} (\vztilde_{k} - \vz_{k})} \\
    &= - 2 \sqrt{\eta \mu} \inner{\hat{\oS} (\vzhat_{k}), \vztilde_{k} - \vz_{\star} + \bigopen{\frac{1}{\sqrt{\eta \mu}} - 1} (\vztilde_{k} - \vztilde_{k-1})} \\
    &\phantom{={}} - 2 \sqrt{\eta \mu} \inner{\hat{\oS} (\vzhat_{k}), \bigopen{\frac{1}{\sqrt{\eta \mu}} - 1} (\vztilde_{k-1} - \vz_{k})} + 2 \eta \langle \hat{\oS} (\vzhat_{k}),  \nabla \hat{\phi} (\vztilde_{k}) \rangle \\
    &= - 2 \sqrt{\eta \mu} \langle \hat{\oS} (\vzhat_{k}), \vzhat_{k} - \vz_{\star} \rangle \\
    &\phantom{={}} - 2 \eta \bigopen{1 - \sqrt{\eta \mu}} \langle \hat{\oS} (\vzhat_{k}),\nabla \hat{\phi} (\vztilde_{k-1}) + \hat{\oS} (\vzhat_{k-1}) \rangle + 2 \eta \langle \hat{\oS} (\vzhat_{k}),  \nabla \hat{\phi} (\vztilde_{k}) \rangle,
\end{align*}
where we use $\vz_{k} = \vztilde_{k-1} - \eta ( \nabla \hat{\phi} (\vztilde_{k-1}) + \hat{\oS} (\vzhat_{k-1}) )$.
Hence \eqref{eq:endofstepone} is equivalent to
\begin{align*}
    &\begin{aligned}
    &\mu \bignorm{\vztilde_{k+1} - \vz_{\star} + \frac{1}{\sqrt{\eta \mu}} (\vztilde_{k+1} - \vz_{k+1})}^2 + 2 \hat{\phi} (\vztilde_{k}) \\
    &\le \bigopen{1 - \sqrt{\eta \mu}} \bigopen{\mu \bignorm{\vztilde_{k} - \vz_{\star} + \frac{1}{\sqrt{\eta \mu}} (\vztilde_{k} - \vz_{k})}^2 + 2 \hat{\phi} (\vztilde_{k-1})} \\
    &\phantom{\le{}} - 2 \eta \bigopen{1 - \sqrt{\eta \mu}} \langle \nabla \hat{\phi} (\vztilde_{k-1}), \nabla \hat{\phi} (\vztilde_{k-1}) + \hat{\oS} (\vzhat_{k-1}) \rangle \\
    &\phantom{\le{}} + \eta^{2} L_{\phi} \bigopen{1 - \sqrt{\eta \mu}} \| \nabla \hat{\phi} (\vztilde_{k-1}) + \hat{\oS} (\vzhat_{k-1}) \|^{2} \\
    &\phantom{\le{}} - \bigopen{1 - \eta \mu} \sqrt{\frac{\mu}{\eta}} \bignorm{\vztilde_{k} - \vz_{k}}^2 + \eta \| \nabla \hat{\phi} (\vztilde_{k}) \|^2 + \eta \| \hat{\oS} (\vzhat_{k}) \|^2 \\
    &\phantom{\le{}} - 2 \eta \bigopen{1 - \sqrt{\eta \mu}} \langle \hat{\oS} (\vzhat_{k}),\nabla \hat{\phi} (\vztilde_{k-1}) + \hat{\oS} (\vzhat_{k-1}) \rangle + 2 \eta \langle \hat{\oS} (\vzhat_{k}),  \nabla \hat{\phi} (\vztilde_{k}) \rangle \\
    &\phantom{\le{}} - 2 \sqrt{\eta \mu} \langle \hat{\oS} (\vzhat_{k}), \vzhat_{k} - \vz_{\star} \rangle.
    \end{aligned} 
\end{align*}

By monotonicity of $\hat{\oS}$ and $\hat{\oS} (\vz_{\star}) = \mathbf{0}$, we have $- 2 \sqrt{\eta \mu} \langle \hat{\oS} (\vzhat_{k}), \vzhat_{k} - \vz_{\star} \rangle \le 0$, and we can rearrange and complete the squares to obtain:    
\begin{align*}
    &\begin{aligned}
    &\mu \bignorm{\vztilde_{k+1} - \vz_{\star} + \frac{1}{\sqrt{\eta \mu}} (\vztilde_{k+1} - \vz_{k+1})}^2 + 2 \hat{\phi} (\vztilde_{k}) \\
    &\le \bigopen{1 - \sqrt{\eta \mu}} \bigopen{\mu \bignorm{\vztilde_{k} - \vz_{\star} + \frac{1}{\sqrt{\eta \mu}} (\vztilde_{k} - \vz_{k})}^2 + 2 \hat{\phi} (\vztilde_{k-1})} \\
    &\phantom{\le{}} - 2 \eta \bigopen{1 - \sqrt{\eta \mu}} \| \nabla \hat{\phi} (\vztilde_{k-1}) + \hat{\oS} (\vzhat_{k-1}) \|^2 \\
    &\phantom{\le{}} + \eta^{2} L_{\phi} \bigopen{1 - \sqrt{\eta \mu}} \| \nabla \hat{\phi} (\vztilde_{k-1}) + \hat{\oS} (\vzhat_{k-1}) \|^{2} \\
    &\phantom{\le{}} - \bigopen{1 - \eta \mu} \sqrt{\frac{\mu}{\eta}} \bignorm{\vztilde_{k} - \vz_{k}}^2 + \eta \| \nabla \hat{\phi} (\vztilde_{k}) + \hat{\oS} (\vzhat_{k}) \|^2 \\
    &\phantom{\le{}} + 2 \eta \bigopen{1 - \sqrt{\eta \mu}} \langle \hat{\oS} (\vzhat_{k-1}) - \hat{\oS} (\vzhat_{k}),\nabla \hat{\phi} (\vztilde_{k-1}) + \hat{\oS} (\vzhat_{k-1}) \rangle.
    \end{aligned}
\end{align*}
For later use, we rearrange again and insert some additional terms as follows:
\begin{align*}
    &\begin{aligned}
    &\mu \bignorm{\vztilde_{k+1} - \vz_{\star} + \frac{1}{\sqrt{\eta \mu}} (\vztilde_{k+1} - \vz_{k+1})}^2 + 2 \hat{\phi} (\vztilde_{k}) \\
    &- \eta \| \nabla \hat{\phi} (\vztilde_{k}) + \hat{\oS} (\vzhat_{k}) \|^2 + \bigopen{1 - \eta \mu} \sqrt{\frac{\mu}{\eta}} \bignorm{\vztilde_{k} - \vz_{k}}^2 \\
    &+ B \eta \bigopen{1 - \sqrt{\eta \mu}} \bigopen{1 - \eta L_{\phi}} \| \nabla \hat{\phi} (\vztilde_{k-1}) + \hat{\oS} (\vzhat_{k-1}) \|^2 \\
    &\le \bigopen{1 - \sqrt{\eta \mu}} \bigopen{\mu \bignorm{\vztilde_{k} - \vz_{\star} + \frac{1}{\sqrt{\eta \mu}} (\vztilde_{k} - \vz_{k})}^2 + 2 \hat{\phi} (\vztilde_{k-1})} \\
    &\phantom{\le{}} + \bigopen{1 - \sqrt{\eta \mu}} \bigopen{- \eta  \| \nabla \hat{\phi} (\vztilde_{k-1}) + \hat{\oS} (\vzhat_{k-1}) \|^2 + \bigopen{1 - \eta \mu} \sqrt{\frac{\mu}{\eta}} \bignorm{\vztilde_{k-1} - \vz_{k-1}}^2} \\
    &\phantom{\le{}} + \bigopen{1 - \sqrt{\eta \mu}} \bigopen{B \eta \bigopen{1 - \sqrt{\eta \mu}} \bigopen{1 - \eta L_{\phi}} \| \nabla \hat{\phi} (\vztilde_{k-2}) + \hat{\oS} (\vzhat_{k-2}) \|^2} \\
    &\phantom{\le{}} - (1 - B) \eta \bigopen{1 - \sqrt{\eta \mu}} \bigopen{1 - \eta L_{\phi}} \| \nabla \hat{\phi} (\vztilde_{k-1}) + \hat{\oS} (\vzhat_{k-1}) \|^2 \\
    &\phantom{\le{}} - B \eta \bigopen{1 - \sqrt{\eta \mu}}^2 \bigopen{1 - \eta L_{\phi}} \| \nabla \hat{\phi} (\vztilde_{k-2}) + \hat{\oS} (\vzhat_{k-2}) \|^2 \\
    &\phantom{\le{}} - \bigopen{1 - \sqrt{\eta \mu}} \bigopen{1 - \eta \mu} \sqrt{\frac{\mu}{\eta}} \bignorm{\vztilde_{k-1} - \vz_{k-1}}^2 \\
    &\phantom{\le{}} + 2 \eta \bigopen{1 - \sqrt{\eta \mu}} \langle \hat{\oS} (\vzhat_{k-1}) - \hat{\oS} (\vzhat_{k}),\nabla \hat{\phi} (\vztilde_{k-1}) + \hat{\oS} (\vzhat_{k-1}) \rangle.
    \end{aligned} 
\end{align*}
(Note that $\eta \le C/L_{\phi} < 1 / L_{\phi}$ implies $1 - \eta L_{\phi} \ge 1 - C > 0$.)

Recalling the definition of $\Psi_{k}$, we obtain
\begin{align}
    \begin{aligned}
    \Psi_{k+1} &\le \bigopen{1 - \sqrt{\eta \mu}} \Psi_{k} \\
    &\phantom{\le{}} - (1 - B) \eta \bigopen{1 - \sqrt{\eta \mu}} \bigopen{1 - \eta L_{\phi}} \| \nabla \hat{\phi} (\vztilde_{k-1}) + \hat{\oS} (\vzhat_{k-1}) \|^2 \\
    &\phantom{\le{}} - B \eta \bigopen{1 - \sqrt{\eta \mu}}^2 \bigopen{1 - \eta L_{\phi}} \| \nabla \hat{\phi} (\vztilde_{k-2}) + \hat{\oS} (\vzhat_{k-2}) \|^2 \\
    &\phantom{\le{}} - \bigopen{1 - \sqrt{\eta \mu}} \bigopen{1 - \eta \mu} \sqrt{\frac{\mu}{\eta}} \bignorm{\vztilde_{k-1} - \vz_{k-1}}^2 \\
    &\phantom{\le{}} + 2 \eta \bigopen{1 - \sqrt{\eta \mu}} \langle \hat{\oS} (\vzhat_{k-1}) - \hat{\oS} (\vzhat_{k}),\nabla \hat{\phi} (\vztilde_{k-1}) + \hat{\oS} (\vzhat_{k-1}) \rangle.
    \end{aligned} \label{eq:contractionone}
\end{align}

Observe that
\begin{align*}
    \vzhat_{k} &= \vztilde_{k-1} + \frac{1}{\sqrt{\eta \mu}} (\vztilde_{k} - \vztilde_{k-1}), \\
    \vzhat_{k-1} &= \vztilde_{k-1} + \bigopen{\frac{1}{\sqrt{\eta \mu}} - 1} (\vztilde_{k-1} - \vztilde_{k-2}).
\end{align*}
Therefore we have
\begin{align*}
    \vzhat_{k-1} - \vzhat_{k} &= \vzhat_{k-1} - \vztilde_{k-1} - \frac{1}{\sqrt{\eta \mu}} (\vztilde_{k} - \vztilde_{k-1}) \\
    &= \bigopen{\frac{1}{\sqrt{\eta \mu}} - 1} (\vztilde_{k-1} - \vztilde_{k-2}) - \frac{1}{\sqrt{\eta \mu}} (\vztilde_{k} - \vztilde_{k-1}) \\
    &= \bigopen{\frac{1}{\sqrt{\eta \mu}} - 1} (\vztilde_{k-1} - \vz_{k-1}) + \bigopen{\frac{1}{\sqrt{\eta \mu}} - 1} (\vz_{k-1} - \vztilde_{k-2}) \\
    &\phantom{={}} - \frac{1}{\sqrt{\eta \mu}} (\vztilde_{k} - \vz_{k}) - \frac{1}{\sqrt{\eta \mu}} (\vz_{k} - \vztilde_{k-1}).
\end{align*}
Note that
\begin{align*}
    \vz_{k-1} - \vztilde_{k-2} &= - \eta ( \nabla \hat{\phi} (\vztilde_{k-2}) + \hat{\oS} (\vzhat_{k-2}) ), \\
    \vz_{k} - \vztilde_{k-1} &= - \eta ( \nabla \hat{\phi} (\vztilde_{k-1}) + \hat{\oS} (\vzhat_{k-1}) ), \\
    \vztilde_{k} - \vz_{k} &= \frac{1 - \sqrt{\eta \mu}}{1 + \sqrt{\eta \mu}} \bigopen{\vz_{k} - \vz_{k-1}} \\
    &= \frac{1 - \sqrt{\eta \mu}}{1 + \sqrt{\eta \mu}} ( \vztilde_{k-1} - \vz_{k-1} - \eta ( \nabla \hat{\phi} (\vztilde_{k-1}) + \hat{\oS} (\vzhat_{k-1}) ) )
\end{align*}
and therefore
\begin{align*}
    \vzhat_{k-1} - \vzhat_{k} &= \bigopen{\frac{1}{\sqrt{\eta \mu}} - 1} (\vztilde_{k-1} - \vz_{k-1}) - \eta \bigopen{\frac{1}{\sqrt{\eta \mu}} - 1} ( \nabla \hat{\phi} (\vztilde_{k-2}) + \hat{\oS} (\vzhat_{k-2}) ) \\
    &\phantom{={}} - \frac{1}{\sqrt{\eta \mu}} \cdot \frac{1 - \sqrt{\eta \mu}}{1 + \sqrt{\eta \mu}} ( \vztilde_{k-1} - \vz_{k-1} - \eta ( \nabla \hat{\phi} (\vztilde_{k-1}) + \hat{\oS} (\vzhat_{k-1}) ) ) \\
    &\phantom{={}} + \sqrt{\frac{\eta}{\mu}} ( \nabla \hat{\phi} (\vztilde_{k-1}) + \hat{\oS} (\vzhat_{k-1}) ) \\
    &= \frac{1 - \sqrt{\eta \mu}}{1 + \sqrt{\eta \mu}} \bigopen{\vztilde_{k-1} - \vz_{k-1}} - \bigopen{1 - \sqrt{\eta \mu}} \sqrt{\frac{\eta}{\mu}} ( \nabla \hat{\phi} (\vztilde_{k-2}) + \hat{\oS} (\vzhat_{k-2}) ) \\
    &\phantom{={}} + \frac{2}{1 + \sqrt{\eta \mu}} \sqrt{\frac{\eta}{\mu}} ( \nabla \hat{\phi} (\vztilde_{k-1}) + \hat{\oS} (\vzhat_{k-1}) ).
\end{align*}
By Lipschitzness of $\hat{\oS}$ and the triangle inequality, we have
\begin{align*}
    \| \hat{\oS}(\vzhat_{k-1}) - \hat{\oS}(\vzhat_{k}) \| &\le L_{\oS} \bignorm{\vzhat_{k-1} - \vzhat_{k}} \\
    &\le \frac{1 - \sqrt{\eta \mu}}{1 + \sqrt{\eta \mu}} L_{\oS} \bignorm{\vztilde_{k-1} - \vz_{k-1}} \\
    &\phantom{\le{}} + \bigopen{1 - \sqrt{\eta \mu}} \sqrt{\frac{\eta}{\mu}} L_{\oS} \| \nabla \hat{\phi} (\vztilde_{k-2}) + \hat{\oS} (\vzhat_{k-2}) \| \\
    &\phantom{\le{}} + \frac{2}{1 + \sqrt{\eta \mu}} \sqrt{\frac{\eta}{\mu}} L_{\oS} \| \nabla \hat{\phi} (\vztilde_{k-1}) + \hat{\oS} (\vzhat_{k-1}) \|.
\end{align*}
Then, we have
\begin{align*}
    &2 \eta \bigopen{1 - \sqrt{\eta \mu}} \langle \hat{\oS} (\vzhat_{k-1}) - \hat{\oS} (\vzhat_{k}), \nabla \hat{\phi} (\vztilde_{k-1}) + \hat{\oS} (\vzhat_{k-1}) \rangle \\
    &\le 2 \eta \bigopen{1 - \sqrt{\eta \mu}} \| \hat{\oS} (\vzhat_{k-1}) - \hat{\oS} (\vzhat_{k}) \| \cdot \| \nabla \hat{\phi} (\vztilde_{k-1}) + \hat{\oS} (\vzhat_{k-1}) \| \\
    &\le 2 \eta \bigopen{1 - \sqrt{\eta \mu}} \frac{1 - \sqrt{\eta \mu}}{1 + \sqrt{\eta \mu}} L_{\oS} \bignorm{\vztilde_{k-1} - \vz_{k-1}} \cdot \| \nabla \hat{\phi} (\vztilde_{k-1}) + \hat{\oS} (\vzhat_{k-1}) \| \\
    &\phantom{\le{}} + 2 \eta \bigopen{1 - \sqrt{\eta \mu}}^2 \sqrt{\frac{\eta}{\mu}} L_{\oS} \| \nabla \hat{\phi} (\vztilde_{k-2}) + \hat{\oS} (\vzhat_{k-2}) \| \cdot \| \nabla \hat{\phi} (\vztilde_{k-1}) + \hat{\oS} (\vzhat_{k-1}) \| \\
    &\phantom{\le{}} + 4 \eta \frac{1 - \sqrt{\eta \mu}}{1 + \sqrt{\eta \mu}} \sqrt{\frac{\eta}{\mu}} L_{\oS} \| \nabla \hat{\phi} (\vztilde_{k-1}) + \hat{\oS} (\vzhat_{k-1}) \|^2.
\end{align*}

Now we can apply the above to \eqref{eq:contractionone}, which yields
\begin{align}
    \begin{aligned}
    \Psi_{k+1} &\le \bigopen{1 - \sqrt{\eta \mu}} \Psi_{k} \\
    &\phantom{\le{}} - (1 - B) \eta \bigopen{1 - \sqrt{\eta \mu}} \bigopen{1 - \eta L_{\phi}} \| \nabla \hat{\phi} (\vztilde_{k-1}) + \hat{\oS} (\vzhat_{k-1}) \|^2 \\
    &\phantom{\le{}} - B \eta \bigopen{1 - \sqrt{\eta \mu}}^2 \bigopen{1 - \eta L_{\phi}} \| \nabla \hat{\phi} (\vztilde_{k-2}) + \hat{\oS} (\vzhat_{k-2}) \|^2 \\
    &\phantom{\le{}} - \bigopen{1 - \sqrt{\eta \mu}} \bigopen{1 - \eta \mu} \sqrt{\frac{\mu}{\eta}} \bignorm{\vztilde_{k-1} - \vz_{k-1}}^2 \\
    &\phantom{\le{}} + 2 \eta \bigopen{1 - \sqrt{\eta \mu}} \frac{1 - \sqrt{\eta \mu}}{1 + \sqrt{\eta \mu}} L_{\oS} \bignorm{\vztilde_{k-1} - \vz_{k-1}} \cdot \| \nabla \hat{\phi} (\vztilde_{k-1}) + \hat{\oS} (\vzhat_{k-1}) \| \\
    &\phantom{\le{}} + 2 \eta \bigopen{1 - \sqrt{\eta \mu}}^2 \sqrt{\frac{\eta}{\mu}} L_{\oS} \| \nabla \hat{\phi} (\vztilde_{k-2}) + \hat{\oS} (\vzhat_{k-2}) \| \cdot \| \nabla \hat{\phi} (\vztilde_{k-1}) + \hat{\oS} (\vzhat_{k-1}) \| \\
    &\phantom{\le{}} + 4 \eta \frac{1 - \sqrt{\eta \mu}}{1 + \sqrt{\eta \mu}} \sqrt{\frac{\eta}{\mu}} L_{\oS} \| \nabla \hat{\phi} (\vztilde_{k-1}) + \hat{\oS} (\vzhat_{k-1}) \|^2.
    \end{aligned} \label{eq:contractiontwo}
\end{align}

Now we will show that the sum of the remainder terms in \eqref{eq:contractiontwo} can be bounded above by zero.
Recall that we choose $B = \frac{\sqrt{C}}{1 - C}$ and $C$ such that
\begin{align*}
    1 - C - \sqrt{C} \bigopen{C + 6} = 0.
\end{align*}

First, since
\begin{align*}
    \eta \le C \cdot \min \bigset{\frac{1}{L_{\phi}}, \frac{\mu}{L_{\oS}^2}},
\end{align*}
we have $\sqrt{\eta \mu} \le \sqrt{\eta L_{\phi}} \le \sqrt{C}$ and $\frac{\eta}{\mu} L_{\oS}^2 \le C$.
Therefore we can compute
\begin{align*}
    \sqrt{\eta \mu} \cdot \frac{\eta}{\mu} L_{\oS}^2 \le C \sqrt{C} &= \bigopen{1 - 6B} \bigopen{1 - C} \\
    &\le \bigopen{1 - 6B} \bigopen{1 - \eta L_{\phi}} \le \bigopen{1 - 6B} \cdot \frac{\bigopen{1 + \sqrt{\eta \mu}}^{3} \bigopen{1 - \eta L_{\phi}}}{1 - \sqrt{\eta \mu}}.
\end{align*}
Here, we use
\begin{align*}
    \bigopen{1 - 6B} \bigopen{1 - C} &= 1 - C - 6 B(1-C) = 1 - C - 6 \sqrt{C} = C \sqrt{C}.
\end{align*}
Thus we have
\begin{align*}
    &\bigopen{\eta \bigopen{1 - \sqrt{\eta \mu}} \frac{1 - \sqrt{\eta \mu}}{1 + \sqrt{\eta \mu}} L_{\oS}}^2 \\
    &\le \bigopen{1 - \sqrt{\eta \mu}} \bigopen{1 - \eta \mu} \sqrt{\frac{\mu}{\eta}} \cdot \bigopen{1 - 6B} \eta \bigopen{1 - \sqrt{\eta \mu}} \bigopen{1 - \eta L_{\phi}}
\end{align*}
and therefore
\begin{align}
    \begin{aligned}
    &2 \eta \bigopen{1 - \sqrt{\eta \mu}} \frac{1 - \sqrt{\eta \mu}}{1 + \sqrt{\eta \mu}} L_{\oS} \bignorm{\vztilde_{k-1} - \vz_{k-1}} \cdot \| \nabla \hat{\phi} (\vztilde_{k-1}) + \hat{\oS} (\vzhat_{k-1}) \| \\
    &\le \bigopen{1 - \sqrt{\eta \mu}} \bigopen{1 - \eta \mu} \sqrt{\frac{\mu}{\eta}} \bignorm{\vztilde_{k-1} - \vz_{k-1}}^2 \\
    &\phantom{\le{}} + \bigopen{1 - 6B} \eta \bigopen{1 - \sqrt{\eta \mu}} \bigopen{1 - \eta L_{\phi}} \| \nabla \hat{\phi} (\vztilde_{k-1}) + \hat{\oS} (\vzhat_{k-1}) \|^2.
    \end{aligned} \label{eq:amgmone}
\end{align}

Second, we can compute
\begin{align*}
    \frac{\eta}{\mu} L_{\oS}^2 &\le C = B^2 (1-C)^2 \le B^2 \bigopen{1 - \eta L_{\phi}}^2 \le B^2 \frac{\bigopen{1 - \eta L_{\phi}}^2}{1 - \sqrt{\eta \mu}}.
\end{align*}
Thus we have
\begin{align*}
    &\bigopen{\eta \bigopen{1 - \sqrt{\eta \mu}}^2 \sqrt{\frac{\eta}{\mu}} L_{\oS}}^2 \le B \eta \bigopen{1 - \sqrt{\eta \mu}}^2 \bigopen{1 - \eta L_{\phi}} \cdot B \eta \bigopen{1 - \sqrt{\eta \mu}} \bigopen{1 - \eta L_{\phi}}
\end{align*}
and therefore
\begin{align}
    \begin{aligned}
    &2 \eta \bigopen{1 - \sqrt{\eta \mu}}^2 \sqrt{\frac{\eta}{\mu}} L_{\oS} \| \nabla \hat{\phi} (\vztilde_{k-2}) + \hat{\oS} (\vzhat_{k-2}) \| \cdot \| \nabla \hat{\phi} (\vztilde_{k-1}) + \hat{\oS} (\vzhat_{k-1}) \| \\
    &\le B \eta \bigopen{1 - \sqrt{\eta \mu}}^2 \bigopen{1 - \eta L_{\phi}} \| \nabla \hat{\phi} (\vztilde_{k-2}) + \hat{\oS} (\vzhat_{k-2}) \|^2 \\
    &\phantom{\le{}} + B \eta \bigopen{1 - \sqrt{\eta \mu}} \bigopen{1 - \eta L_{\phi}} \| \nabla \hat{\phi} (\vztilde_{k-1}) + \hat{\oS} (\vzhat_{k-1}) \|^2.
    \end{aligned}
    \label{eq:amgmtwo}
\end{align}

Lastly, we can compute
\begin{align*}
    \sqrt{\frac{\eta}{\mu}} L_{\oS} &\le \sqrt{C} \le B (1-C) \le B \bigopen{1 - \eta L_{\phi}} \le B \bigopen{1 + \sqrt{\eta \mu}} \bigopen{1 - \eta L_{\phi}}.
\end{align*}
Thus we have
\begin{align*}
    4 \eta \frac{1 - \sqrt{\eta \mu}}{1 + \sqrt{\eta \mu}} \sqrt{\frac{\eta}{\mu}} L_{\oS} \le 4B \eta \bigopen{1 - \sqrt{\eta \mu}} \bigopen{1 - \eta L_{\phi}}
\end{align*}
and therefore
\begin{align}
    \begin{aligned}
        &4 \eta \frac{1 - \sqrt{\eta \mu}}{1 + \sqrt{\eta \mu}} \sqrt{\frac{\eta}{\mu}} L_{\oS} \| \nabla \hat{\phi} (\vztilde_{k-1}) + \hat{\oS} (\vzhat_{k-1}) \|^2 \\
        &\le 4B \eta \bigopen{1 - \sqrt{\eta \mu}} \bigopen{1 - \eta L_{\phi}} \| \nabla \hat{\phi} (\vztilde_{k-1}) + \hat{\oS} (\vzhat_{k-1}) \|^2.
    \end{aligned}
    \label{eq:amgmthree}
\end{align}

Adding \eqref{eq:amgmone}, \eqref{eq:amgmtwo}, and \eqref{eq:amgmthree} to \eqref{eq:contractiontwo}, we have
\begin{align*}
    \Psi_{k+1} &\le \bigopen{1 - \sqrt{\eta \mu}} \Psi_{k},
\end{align*}
which proves the theorem statement.
\end{proof}

Note that the decomposition $\oT = \nabla \phi + \oS$ need not be Asplund, as was also the case in \Cref{thm:ode}. 

The following lemma ensures that the convergence rate of $\Psi_k\rightarrow 0$ directly translates to the convergence rate of $\vztilde_{k} \rightarrow \vz_{\star}$.
Note that $\hat{\phi}(\vz)$ is nonnegative, since it is defined as the Bregman divergence between $\vz$ and $\vz_{\star}$ associated with a convex function $\phi$.

\begin{restatable}{lemma}{thmdiscretizevalid}
\label{lem:discretizevalid}

Under the same assumptions and Lyapunov function as in \cref{thm:discretize}, we have
\begin{align*}
    \Psi_{k} &\ge \frac{\mu}{2} \bignorm{\vztilde_{k} + \frac{1}{\sqrt{\eta \mu}} \bigopen{\vztilde_{k} - \vz_{k}} - \vz_{\star}}^2 + \hat{\phi} (\vztilde_{k-1}) + (1 - \eta \mu) \sqrt{\frac{\mu}{\eta}} \bignorm{\vztilde_{k-1} - \vz_{k-1}}^{2}
\end{align*}
for $k \ge 2$.
\end{restatable}

\begin{proof}
We reuse the shifted function $\hat{\phi}$ and operator $\hat{\oS}$ defined in \eqref{eq:affineshift} and \eqref{eq:affineshift2}.
By $\hat{\phi}(\vz_{\star}) = 0$, $\nabla \hat{\phi}(\vz_{\star}) = \bf{0}$, and convexity and $L_{\phi}$-smoothness of $\hat{\phi}$, we have
\begin{align*}
    \frac{1}{2L_{\phi}} \| \nabla \hat{\phi} (\vztilde_{k-1}) \|^2 \le \hat{\phi} (\vztilde_{k-1}).
\end{align*}
Recall that
\begin{align*}
    &\vztilde_{k} - \vz_{\star} + \frac{1}{\sqrt{\eta \mu}} (\vztilde_{k} - \vz_{k}) \\
    &= \vztilde_{k-1} - \vz_{\star} - \sqrt{\frac{\eta}{\mu}} ( \nabla \hat{\phi} (\vztilde_{k-1}) + \hat{\oS} (\vzhat_{k-1}) ) + \bigopen{\frac{1}{\sqrt{\eta \mu}} - 1} (\vztilde_{k-1} - \vz_{k-1})
\end{align*}
as in \eqref{eq:nag1} and \eqref{eq:nag2} with $k \leftarrow k-1$.
Then we have
\begin{align*}
    &\vzhat_{k-1} - \vz_{\star} - \sqrt{\frac{\eta}{\mu}} ( \nabla \hat{\phi} (\vztilde_{k-1}) + \hat{\oS} (\vzhat_{k-1}) ) + \bigopen{\frac{1}{\sqrt{\eta \mu}} - 1} (\vztilde_{k-2} - \vz_{k-1}) \\
    &= \vzhat_{k-1} - \vz_{\star} - \sqrt{\frac{\eta}{\mu}} ( \nabla \hat{\phi} (\vztilde_{k-1}) + \hat{\oS} (\vzhat_{k-1}) ) + \bigopen{\frac{1}{\sqrt{\eta \mu}} - 1} \eta ( \nabla \hat{\phi} (\vztilde_{k-2}) + \hat{\oS} (\vzhat_{k-2}) ).
\end{align*}
Thus, by the Cauchy-Schwarz inequality and $\frac{1}{\sqrt{\eta \mu}} \ge \sqrt{\frac{1}{C} \cdot \frac{L_{\phi}}{\mu}} \ge \sqrt{\frac{L_{\phi}}{\mu}} \ge 1$ we have
\begin{align*}
    \bignorm{\vzhat_{k-1} - \vz_{\star}}^{2} &\le 4 \bignorm{\vztilde_{k} - \vz_{\star} + \frac{1}{\sqrt{\eta \mu}} (\vztilde_{k} - \vz_{k})}^2 + 4 \cdot \frac{\eta}{\mu} \| \nabla \hat{\phi} (\vztilde_{k-1}) + \hat{\oS} (\vzhat_{k-1}) \|^{2} \\
    &\phantom{{}\le} + 2 \bigopen{\frac{1}{\sqrt{\eta \mu}} - 1}^2 \eta^2 \| \nabla \hat{\phi} (\vztilde_{k-2}) + \hat{\oS} (\vzhat_{k-2}) \|^{2} \\ &\le 4 \bignorm{\vztilde_{k} - \vz_{\star} + \frac{1}{\sqrt{\eta \mu}} (\vztilde_{k} - \vz_{k})}^2 + 4 \cdot \frac{\eta}{\mu} \| \nabla \hat{\phi} (\vztilde_{k-1}) + \hat{\oS} (\vzhat_{k-1}) \|^{2} \\
    &\phantom{{}\le} + 2 \cdot \frac{\eta}{\mu} \| \nabla \hat{\phi} (\vztilde_{k-2}) + \hat{\oS} (\vzhat_{k-2}) \|^{2}.
\end{align*}
By $L_{\phi}$-smoothness of $\hat{\phi}$, $L_{\oS}$-Lipschitzness of $\hat{\oS}$, and $\nabla \hat{\phi}(\vz_{\star}) = \hat{\oS}(\vz_{\star}) = \bm{0}$, we have
\begin{align*}
    &\eta \| \nabla \hat{\phi} (\vztilde_{k-1}) + \hat{\oS} (\vzhat_{k-1}) \|^2 \\
    &\le 2 \eta \| \nabla \hat{\phi} (\vztilde_{k-1}) \|^2 + 2 \eta \| \hat{\oS} (\vzhat_{k-1}) \|^2 \\
    &\le 2 \eta \| \nabla \hat{\phi} (\vztilde_{k-1}) \|^2 + 2 \eta L_{\oS}^2 \bignorm{\vzhat_{k-1} - \vz_{\star}}^2 \\
    &\le 4 \eta L_{\phi} \, \hat{\phi} (\vztilde_{k-1}) + 8 \eta L_{\oS}^2 \bignorm{\vztilde_{k} - \vz_{\star} + \frac{1}{\sqrt{\eta \mu}} (\vztilde_{k} - \vz_{k})}^2 \\
    &\phantom{{}\le} + 8 \frac{\eta^2 L_{\oS}^2}{\mu} \| \nabla \hat{\phi} (\vztilde_{k-1}) + \hat{\oS} (\vzhat_{k-1}) \|^2 + 4 \frac{\eta^2 L_{\oS}^2}{\mu} \| \nabla \hat{\phi} (\vztilde_{k-2}) + \hat{\oS} (\vzhat_{k-2}) \|^2
\end{align*}
and hence
\begin{align}
    \begin{aligned}
    &\eta \bigopen{1 - 8 \frac{\eta}{\mu} L_{\oS}^2} \| \nabla \hat{\phi} (\vztilde_{k-1}) + \hat{\oS} (\vzhat_{k-1}) \|^2 \\
    &\le 4 \eta L_{\phi} \, \hat{\phi} (\vztilde_{k-1}) + 8 \eta L_{\oS}^2 \bignorm{\vztilde_{k} - \vz_{\star} + \frac{1}{\sqrt{\eta \mu}} (\vztilde_{k} - \vz_{k})}^2 \\
    &\phantom{{}\le} + 4 \frac{\eta^2 L_{\oS}^2}{\mu} \| \nabla \hat{\phi} (\vztilde_{k-2}) + \hat{\oS} (\vzhat_{k-2}) \|^2.
    \end{aligned} \label{eq:validity}
\end{align}

Moreover, we can recall that
\begin{align*}
    \sqrt{\eta \mu} &\le \sqrt{\eta L_{\phi}} \le \sqrt{C}, \quad \frac{\eta}{\mu} L_{\oS}^2 \le C.
\end{align*}
Since $C \approx 0.026117 \le \frac{1}{12}$ we have
$4 \eta L_{\phi} \le 4C \le 1 - 8C \le 1 - 8 \frac{\eta}{\mu} L_{\oS}^2$,
and therefore
\begin{align}
    \frac{4 \eta L_{\phi}}{1 - 8 \frac{\eta}{\mu} L_{\oS}^2} &\le 1. \label{eq:validityone}
\end{align}
Since $C \approx 0.026117 \le \frac{1}{24}$ we have
$16 \cdot \frac{\eta}{\mu} L_{\oS}^2 \le 16 C \le 1 - 8C \le 1 - 8 \frac{\eta}{\mu} L_{\oS}^2$, 
and therefore
\begin{align}
    \frac{8 \eta L_{\oS}^2}{1 - 8 \frac{\eta}{\mu} L_{\oS}^2} &\le \frac{\mu}{2}. \label{eq:validitytwo}
\end{align}
Since $C \approx 0.026117$ also satisfies
\begin{align*}
    0.10447 &\approx 4C \le \sqrt{C} (1 - \sqrt{C}) (1 - 8C) \approx 0.10718 
\end{align*}
and $B = \frac{\sqrt{C}}{1-C}$, we have
\begin{align*}
    4 \frac{\eta}{\mu} L_{\oS}^2 \le 4 C &\le \sqrt{C} (1 - \sqrt{C}) (1 - 8C) \\
    &\le B (1 - \sqrt{C}) (1-C) (1 - 8C) \le B \bigopen{1 - \sqrt{\eta \mu}} \bigopen{1 - \eta L_{\phi}} \bigopen{1 - 8 \frac{\eta}{\mu} L_{\oS}^2}
\end{align*}
and therefore
\begin{align}
    \frac{4 \frac{\eta^2 L_{\oS}^2}{\mu}}{1 - 8 \frac{\eta}{\mu} L_{\oS}^2} &\le B \eta \bigopen{1 - \sqrt{\eta \mu}} \bigopen{1 - \eta L_{\phi}}. \label{eq:validitythree}
\end{align}

Hence, applying \eqref{eq:validityone}, \eqref{eq:validitytwo}, and \eqref{eq:validitythree} to \eqref{eq:validity}, we have
\begin{align*}
    \eta \| \nabla \hat{\phi} (\vztilde_{k-1}) + \hat{\oS} (\vzhat_{k-1}) \|^2 &\le \hat{\phi} (\vztilde_{k-1}) + \frac{\mu}{2} \bignorm{\vztilde_{k} - \vz_{\star} + \frac{1}{\sqrt{\eta \mu}} (\vztilde_{k} - \vz_{k})}^2 \\
    &\phantom{{}\le} + B \eta \bigopen{1 - \sqrt{\eta \mu}} \bigopen{1 - \eta L_{\phi}} \| \nabla \hat{\phi} (\vztilde_{k-2}) + \hat{\oS} (\vzhat_{k-2}) \|^2
\end{align*}
which immediately implies
\begin{align*}
    \Psi_{k} &\ge \frac{\mu}{2} \bignorm{\vztilde_{k} + \frac{1}{\sqrt{\eta \mu}} \bigopen{\vztilde_{k} - \vz_{k}} - \vz_{\star}}^2 + \hat{\phi} (\vztilde_{k-1}) + (1 - \eta \mu) \sqrt{\frac{\mu}{\eta}} \bignorm{\vztilde_{k-1} - \vz_{k-1}}^{2}
\end{align*}
as desired.
\end{proof}

\cref{thm:discretize} and Lemma~\ref{lem:discretizevalid} together yield the following iteration complexity.
\begin{corollary}
    \label{cor:discretize}
    For \emph{NOD} \eqref{eq:nod} with $\eta = \Theta (\min \{ \frac{1}{L_{\phi}}, \frac{\mu}{L_{\oS}^{2}} \} )$,
    \begin{align*}
        K &\ge \Theta \bigopen{\sqrt{\frac{L_{\phi}}{\mu} + \frac{L_{\oS}^2}{\mu^2}} \cdot \log \frac{1}{\epsilon}}
    \end{align*}
    implies $\norm{\vztilde_{K} - \vz_{\star}}^2 \le \epsilon$.
\end{corollary}

\subsection{Bilinear coupling case}

Consider the \textit{bilinearly coupled} SCSC function
\[ \mathcal{L}(\vx,\vy)=g(\vx)-h(\vy)+\vx^\top \mM \vy.\]
We can use a decomposition of the saddle operator of $\mathcal{L}$ similarly to Example~\ref{ex:2} and apply NOD.
In particular, we define the \textit{scaled} saddle gradient operator
\begin{align*}
    \oT (\vx, \vy) &\coloneq \begin{bmatrix}
        \frac{1}{\sqrt{\mu_{x}}} \nabla g \bigopen{\frac{1}{\sqrt{\mu_{x}}} \vx} \\
        \frac{1}{\sqrt{\mu_{y}}} \nabla h \bigopen{\frac{1}{\sqrt{\mu_{y}}} \vy}
    \end{bmatrix}
    + \frac{1}{\sqrt{\mu_{x} \mu_{y}}} \begin{bmatrix}
        \mM \vy \\
        - \mM^{\top} \vx
    \end{bmatrix}. 
\end{align*}
We can observe that this operator is of the form $\nabla \phi + \oS$, where
\begin{align*}
    \phi(\vx, \vy) &= g \bigopen{\frac{1}{\sqrt{\mu_{x}}} \vx} + h \bigopen{\frac{1}{\sqrt{\mu_{y}}} \vy}
\end{align*}
is a $1$-strongly convex, $\max \{ \frac{L_x}{\mu_x}, \frac{L_y}{\mu_y} \}$-smooth function, and
\begin{align*}
    \oS(\vx, \vy) &= \frac{1}{\sqrt{\mu_{x} \mu_{y}}} \begin{bmatrix}
        \bm{0} & \mM \\
        - \mM^{\top} & \bm{0}
    \end{bmatrix}
    \begin{bmatrix}
        \vx \\ \vy
    \end{bmatrix}
\end{align*}
is a $\frac{L_{xy}}{\sqrt{\mu_{x} \mu_{y}}}$-Lipschitz operator (as $\| \mM \| \le L_{xy}$).
With $\eta_{x}=\frac{\eta}{\mu_x}$ and $\eta_{y}=\frac{\eta}{\mu_y}$, NOD is
\begin{align}
    \begin{aligned}
    \vx_{k+1} &= \vxtilde_{k} - \eta_{x} ( \nabla_{\vx} \mathcal{L} (\vxtilde_{k}, \vytilde_{k}) + \theta \bigopen{\nabla_{\vx} \mathcal{L} (\vxtilde_{k}, \vytilde_{k}) - \nabla_{\vx} \mathcal{L} (\vxtilde_{k}, \vytilde_{k-1})} ), \\
    \vy_{k+1} &= \vytilde_{k} + \eta_{y} ( \nabla_{\vy} \mathcal{L} (\vxtilde_{k}, \vytilde_{k}) + \theta (\nabla_{\vy} \mathcal{L} (\vxtilde_{k}, \vytilde_{k}) - \nabla_{\vy} \mathcal{L} (\vxtilde_{k-1}, \vytilde_{k})) ), \\
    \vxtilde_{k+1} &= \vx_{k+1} + \tau (\vx_{k+1} - \vx_{k}), \\
    \vytilde_{k+1} &= \vy_{k+1} + \tau (\vy_{k+1} - \vy_{k}),
    \end{aligned}
    \label{eq:bc}
\end{align}
where the simplification uses the fact that the acyclic operator $\oS$ is linear.

Since the decomposition $\oT = \nabla \phi + \oS$ satisfies $\mu = 1$ and
\begin{align*}
    L_{\phi} = \max \bigset{ \frac{L_x}{\mu_x}, \frac{L_y}{\mu_y} }, \quad L_{\oS} = \frac{L_{xy}}{\sqrt{\mu_{x} \mu_{y}}},
\end{align*}
we have the following corollary.

\begin{corollary}
    \label{cor:minimaxrate}
    Suppose that $\mathcal{L} (\vx, \vy) = g(\vx) + \vx^{\top} \mM \vy - h(\vy)$ is a bilinearly coupled function with $\mu_{x}$-strongly-convex and $L_{x}$-smooth $g(\vx)$, $\mu_{y}$-strongly-convex and $L_{y}$-smooth $h(\vy)$, and constant $\mM \in \R^{d_{x} \times d_{y}}$ with $\norm{\mM} \le L_{xy}$.
    Then the iterates of \eqref{eq:bc} with
    $\eta = \Theta (\min \{ \frac{\mu_x}{L_{x}}, \frac{\mu_y}{L_{y}}, \frac{\mu_x \mu_y}{L_{xy}^{2}} \} )$
    satisfy $\mu_x \norm{\vxtilde_K - \vx_{\star}}^2 + \mu_y \norm{\vytilde_K - \vy_{\star}}^2 \le \epsilon$ for
    \begin{align*}
    K\ge 
        \Theta \bigopen{ \sqrt{ \frac{L_{x}}{\mu_{x}} + \frac{L_{xy}^2}{\mu_{x}\mu_{y}} + \frac{L_{y}}{\mu_{y}} } \cdot \log \frac{1}{\epsilon} }.
    \end{align*}
\end{corollary}

This matches both the iteration complexity lower bound by \cite{zhang22lb} and the upper bounds of previously known algorithms by \cite{du2022optimal,jin22,kovalev2022accelerated}.

\subsection{Oracle models and matching lower bound} \label{subsec:discussions}

While NOD considers a decomposition $\oT = \nabla \phi + \oS$ and accesses the two components via \emph{decomposed oracles} $\nabla \phi$ and $\oS$, most prior work instead treats the operator $\oT$ as a whole, which we refer to as the \emph{single-evaluation oracle} model.

For \textit{bilinearly coupled} minimax optimization problems, these two oracle models are, in fact, equivalent. Indeed, for $ \mathcal{L} (\vx,\vy) = g(\vx) - h(\vy) + \vx^\top \mM \vy$, we have
\begin{align*}
    \begin{aligned}
    \nabla g (\vx) &= \nabla_{\vx} \mathcal{L} (\vx, \vy) - \mM \vy = 2 \nabla_{\vx} \mathcal{L} (\vx, \vy) - \nabla_{\vx} \mathcal{L} (\vx, 2 \vy) \\
    \mM \vy &= (\nabla g(\vx) + 2 \mM \vy) - (\nabla g(\vx) + \mM \vy) = \nabla_{\vx} \mathcal{L} (\vx, 2 \vy) - \nabla_{\vx} \mathcal{L} (\vx, \vy),
    \end{aligned} 
\end{align*}
and similarly for the $\vy$-component.
Thus, the decomposed oracle $(\nabla g(\vx), \nabla h(\vy))$ can be obtained from three calls/queries to the single-evaluation oracle $\oT (\vx, \vy)$, $\oT (\vx, 2 \vy)$, and $\oT (2 \vx, \vy)$, and the same holds for $(\mM \vy, -\mM^{\top} \vx)$.
Therefore, both NOD and other algorithms from previous work formulated in terms of decomposed oracles \cite{du2022optimal,jin22,kovalev2022accelerated} are implementable in the single-evaluation-oracle model with the same asymptotic iteration-complexity upper bounds in the bilinear coupling setup.
For general SCSC minimax optimization problems with \emph{non-bilinear} coupling, the above equivalence no longer holds.

Our rate for NOD matches the lower bound of \cite{zhang22lb}, which is constructed using a bilinearly coupled minimax instance.
This establishes that NOD is optimal within the class of algorithms that access decomposed oracles.
However, whether the same rate can be achieved for general non-bilinear coupling problems to match the lower bound of \cite{zhang22lb} under the single-evaluation oracle model is open.
It is possible that future work will either (i) develop improved algorithms achieving the same rate with a single-evaluation oracle, or (ii) establish stronger lower bounds based on constructing hard instances with non-bilinear coupling.

\section{Conclusion}
\label{sec:conclusion}
We introduce \textbf{N}esterov acceleration with \textbf{O}perator \textbf{D}ecomposition (NOD), a method that leverages the classical notion of Asplund-type operator decomposition to extend Nesterov's accelerated gradient descent (NAG) to the broader class of monotone inclusion problems, including minimax optimization. Under decomposed oracle access, NOD achieves an accelerated rate that matches the complexity lower bound for the bilinearly coupled minimax problems established in prior work, as discussed with further nuance in \Cref{subsec:discussions}.

Our approach represents a slight departure from the prevailing framework in the minimax optimization literature, which typically assumes access to a single-evaluation oracle as discussed in \Cref{subsec:discussions}. In this sense, NOD is not a ``pure'' first-order method. Nonetheless, by drawing on classical notions of operator decomposition, our work opens the door to new avenues for designing accelerated algorithms and suggests new research directions in minimax optimization and operator structure theory.

\clearpage
{
	\bibliography{refs}
    \bibliographystyle{plain}
}

\clearpage
\appendix

\appendix

\section{\texorpdfstring{Proof of \Cref{thm:sincoupling}}{Proof of Theorem 2.4}}
\label{app:sincoupling}

We start by establishing some base facts. First
$ \mathcal{L} (x, y) = x^2 - y^2 + \sin x \sin y$
is SCSC and smooth by explicitly examining 
\begin{align*}
    D \oT &=
    \begin{bmatrix}
        2 - \sin x \sin y & \cos x \cos y \\
        - \cos x \cos y & 2 + \sin x \sin y
    \end{bmatrix}
\end{align*}
and noting $\mu_x = \mu_y = 1$, $L_x = L_y = 3$, and $L_{xy} = 1$.
Moreover, 
\begin{align*}
    \left\lVert
    \begin{bmatrix}
        2 - \sin x \sin y & \cos x \cos y \\
        - \cos x \cos y & 2 + \sin x \sin y
    \end{bmatrix}
    \right\rVert
    &= \sqrt{4 + \cos^2 x \cos^2 y} + \left\lvert \sin x \sin y \right\rvert \le 3,
\end{align*}
so $\mathcal{L}$ is (jointly) $3$-smooth in $(x, y)$.
Next, note that $\phi$ is convex since $h$ is convex, which follows from $h''=2-|\sin\cdot|\ge 0$. Also, $\oS$ is $2$-Lipschitz as it is the sum of two $1$-Lipschitz operators
\begin{align*}
    \oS (x, y)  &= \begin{bmatrix}
        \cos x \sin y \\
        - \sin x \cos y
    \end{bmatrix}
    +
    \begin{bmatrix}
        \int_0^x \lvert \sin t \rvert 
        \phantom{ \frac{\frac{}{}}{\frac{}{}} }\!\!
        \dd t \\
        \int_0^y \lvert \sin t \rvert
        \phantom{ \frac{\frac{}{}}{\frac{}{}} }\!\!
        \dd t
    \end{bmatrix},
\end{align*}
and $\oS$ is single-valued and monotone as
\begin{align*}
\mathrm{symm}\big(D\oS(x,y)\big) &= \frac{D\oS(x,y)  + D\oS^{\top}(x,y)}{2} = \begin{bmatrix}
     \lvert \sin x \rvert - \sin x \sin y \!\!\!\!\!\!\!\!& 0 \\
        0 & \!\!\!\!\!\!\!\!\lvert \sin y \rvert + \sin x \sin y \
    \end{bmatrix} \succeq 0.
\end{align*}

We want to show that $\oT = \nabla \phi + \oS$ is an Asplund decomposition, and it suffices to show that $\oS$ is acyclic, i.e., whenever there exists a convex function $k(x, y)$ such that $\oS - \partial k$ is monotone, $k$ must be affine.
We first show the following lemma that ensures that $k$ is smooth.

\begin{lemma}
    \label{lem:kissmooth}
    Suppose that $\oA\colon\mathbb{R}^d
    \rightarrow\mathbb{R}^d$ is an $L$-Lipschitz, monotone operator on $\R^{d}$, $k\colon\mathbb{R}^d\rightarrow\mathbb{R}$ is CCP, and $\oA - \partial k$ is monotone.
    Then the function $k$ is $L$-smooth.
\end{lemma}
\begin{proof}
    First, if $z_1, z_2 \in \R^{d}$,
    $g_1 \in \partial k (z_1)$, and $g_2 \in \partial k (z_2)$, then, 
    \begin{align*}
        0\stackrel{\text{(i)}}{\le} \langle g_1 - g_2, z_1 - z_2 \rangle \stackrel{\text{(ii)}}{\le} \langle \oA (z_1) - \oA (z_2), z_1 - z_2 \rangle \stackrel{\text{(iii)}}{\le} L \norm{z_1 - z_2}^2,
    \end{align*}
    where (i) follows from convexity of $k$, (ii) follows from monotonicity of $\oA-\partial k$, and (iii) follows from $L$-Lipschitzness of $\oA$.
    
    This implies $\partial k$ is single-valued by the following reasoning. Suppose $g_1, g_1' \in \partial k(z)$ and $g_2 \in \partial k(z + tv)$ for $v \in \R^{d}$ with $\norm{v} = 1$ and $t > 0$. Then, 
    \begin{align*}
        0 &\le \langle g_2-g_1, tv \rangle \le Lt^2, \qquad
        0 \le \langle g_2-g_1', tv \rangle \le Lt^2,
    \end{align*}
    and therefore
    \begin{align*}
        -Lt &\le \langle g_1-g_1', v \rangle \le Lt.
    \end{align*}
    Since this holds for all $t > 0$ and $v$ with $\| v \| = 1$, we must have $g_1 = g_1'$. Since single-valued, we can write $\partial k =\nabla k$.

    Finally, by Theorem~2.1.5 of \cite{nesterov04}, 
    \[
    0 \le \langle \nabla k(z_1) - \nabla k(z_2), z_1 - z_2 \rangle \le L \norm{z_1 - z_2}^2,
    \]
    implies $k$ is $L$-smooth.
\end{proof}

Suppose that there exists a convex $k\colon \mathbb{R}^2\rightarrow\mathbb{R}$ such that $\oS - \nabla k$ is monotone. 
(Note that $k\in C^1$ by \Cref{lem:kissmooth}, which is why we write $\nabla k$ instead of $\partial k$.)
We aim to show that $\nabla k$ must be constant everywhere.
We first prove this for the case $k \in C^2$, and then we extend the argument to $k \notin C^2$ using a mollifier argument.
\begin{figure}
\centering
\begin{tikzpicture}[scale=0.35,>=stealth,every node/.style={font=\footnotesize}]

\draw[->,thick] ({-2.5*pi},0) -- ({2.5*pi},0) node[right] {$x$};
\draw[->,thick] (0,{-2.5*pi}) -- (0,{2.5*pi}) node[above] {$y$};

\foreach \k/\lab in {-2/{-2\pi}, -1/{-\pi}, 1/{\pi}, 2/{2\pi}} {
  \draw ({\k*pi},0) -- ({\k*pi},0);
  \node[below left=0.5pt] at ({\k*pi},0) {$\lab$};
}

\foreach \k/\lab in {-2/{-2\pi}, -1/{-\pi}, 1/{\pi}, 2/{2\pi}} {
  \draw (0,{\k*pi}) -- (0,{\k*pi});
  \node[below left=0.5pt] at (0,{\k*pi}) {$\lab$};
}

\foreach \k in {-2,-1,0,1,2} {
  \draw[blue,very thick] ({-2.1*pi}, {\k*pi}) -- ({2.1*pi}, {\k*pi});
}

\foreach \k in {-2,-1,0,1,2} {
  \draw[red,very thick] ({\k*pi}, {-2.1*pi}) -- ({\k*pi}, {2.1*pi});
}

\foreach \i in {-2,-1,0,1} {
  \foreach \j in {-2,-1,0,1} {
    \pgfmathtruncatemacro{\parity}{mod(\i+\j,2)}
    \ifnum\parity=0
      \draw[red,very thick]
        ({\i*pi},{(\j+0.5)*pi}) -- ({(\i+1)*pi},{(\j+0.5)*pi});
    \else
      \draw[blue,very thick]
        ({(\i+0.5)*pi},{\j*pi}) -- ({(\i+0.5)*pi},{(\j+1)*pi});
    \fi
  }
}

\node[below left] at (0,0) {$0$};

\node[] at (12,3) {$\nabla^2 k (x, y) = {\color{red}\begin{bmatrix}
        0 & 0 \\ 0 & *
    \end{bmatrix}} $};

\node[] at (12,-3) {$\nabla^2 k (x, y) = {\color{blue}\begin{bmatrix}
        * & 0 \\ 0 & 0
    \end{bmatrix}} $};

\end{tikzpicture}
\caption{Points at which the Hessian $\nabla^2 k(x, y)$ has at most one nonzero entry in.
The {\color{red}\textbf {red}} points satisfy \eqref{eq:redlines} and the {\color{blue}\textbf {blue}} points satisfy \eqref{eq:bluelines}.}
\label{fig:redandblue}
\end{figure}
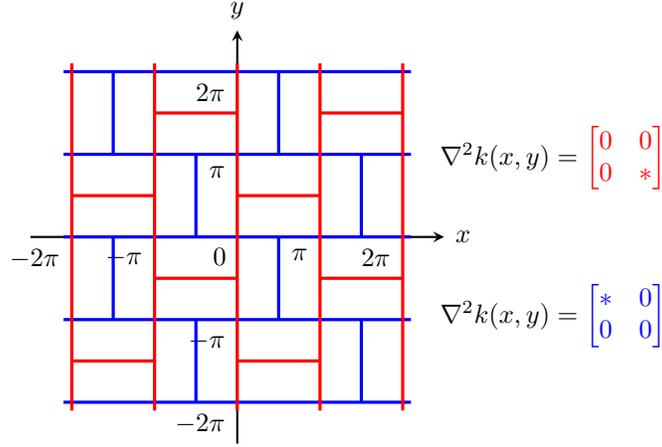

\paragraph{Case 1.}
Assume $k\in C^2$.
Since $\oS - \nabla k$ is monotone, 
\begin{align*}
    &\mathrm{symm}\big(D\oS\big) - \nabla^2 k = \begin{bmatrix}
        \lvert \sin x \rvert - \sin x \sin y - \partial_x^2 k & - \partial_x \partial_y k\\
        - \partial_x \partial_y k & 
      \lvert \sin y \rvert + \sin x \sin y - \partial_y^2 k
    \end{bmatrix} \succeq 0.
\end{align*}

Consider \cref{fig:redandblue} and note that the {\color{red}\textbf {red}} points satisfy $\lvert \sin x \rvert - \sin x \sin y = 0$ which immediately requires $\partial_x^2 k(x, y) =\partial_x \partial_y k(x, y) = 0$ to ensure $\mathrm{symm}\big(D\oS\big)  - \nabla^2 k \succeq 0$. A similar argument also holds for the {\color{blue}\textbf {blue}} points.
Therefore, $k$ satisfies
\begin{align}
    \nabla^2 k (x, y) &= {\begin{bmatrix}
        0 & 0 \\ 0 & *
    \end{bmatrix}} \quad     {\color{red}
\text{if either } 
    \begin{cases}
        x = n \pi \\
        x \in [2n \pi, (2n+1) \pi], y = \bigopen{2m + \frac12} \pi \\
        x \in [(2n-1) \pi, 2n \pi], y = \bigopen{2m - \frac12} \pi
    \end{cases}}
    \label{eq:redlines} \\
    \nabla^2 k (x, y) &= {\begin{bmatrix}
        * & 0 \\ 0 & 0
    \end{bmatrix}} \quad     {\color{blue}
\text{if either } \begin{cases}
        y = m \pi \\
        y \in [2m \pi, (2m+1) \pi], x = \bigopen{2n - \frac12} \pi \\
        y \in [(2m-1) \pi, 2m \pi], x = \bigopen{2n + \frac12} \pi
    \end{cases}}
    \label{eq:bluelines}
\end{align}
where $n, m \in \sZ$, and $*$ denotes possibly nonzero entries.

Since $\nabla k $ is continuously differentiable, we can use the fundamental theorem of calculus to get
\begin{align*}
    \nabla k \left(\frac{\pi}{2}, 0\right) - \nabla k \left(0, 0\right) &= \int_0^{\frac{\pi}{2}} \nabla^2 k (t, 0) \begin{bmatrix}
        1 \\ 0
    \end{bmatrix} \dd t 
    =
    \int_0^{\frac{\pi}{2}} 
    {\color{blue}\begin{bmatrix}
        * & 0 \\ 0 & 0
    \end{bmatrix}}
    \begin{bmatrix}
        1 \\ 0
    \end{bmatrix} \dd t 
    = \begin{bmatrix}
        *  \\ 0
    \end{bmatrix},  \\
    \nabla k \left(\frac{\pi}{2}, -\pi\right) - \nabla k \left(\frac{\pi}{2}, 0\right) &= \int_0^{-\pi} \nabla^2 k \left(\frac{\pi}{2}, t\right) \begin{bmatrix}
        0 \\ 1
    \end{bmatrix} \dd t 
    = \int_0^{-\pi} 
    {\color{blue}\begin{bmatrix}
        * & 0 \\ 0 & 0
    \end{bmatrix}} \begin{bmatrix}
        0 \\ 1
    \end{bmatrix} \dd t 
    = \begin{bmatrix}
        0  \\ 0
    \end{bmatrix},  \\
    \nabla k \left(0, -\pi\right) - \nabla k \left(\frac{\pi}{2}, -\pi\right) &= \int_{\frac{\pi}{2}}^{0} \nabla^2 k \left(t, -\pi\right) \begin{bmatrix}
        1 \\ 0
    \end{bmatrix} \dd t = \int_{\frac{\pi}{2}}^{0} 
    {\color{blue}\begin{bmatrix}
        * & 0 \\ 0 & 0
    \end{bmatrix}}
    \begin{bmatrix}
        1 \\ 0
    \end{bmatrix} \dd t = \begin{bmatrix}
        *  \\ 0
    \end{bmatrix}.
\end{align*}

The sum of these three line integrals yields
\begin{align}
    \nabla k \left(0, -\pi\right) - \nabla k \left(0, 0\right) &= \begin{bmatrix}
        *  \\ 0
    \end{bmatrix}. \label{eq:fot1}
\end{align}
However, we also have
\begin{align}
    \nabla k \left(0, -\pi\right) - \nabla k \left(0, 0\right) &= \int_0^{- \pi} \nabla^2 k (0, t) 
    \begin{bmatrix}
        0 \\ 1
    \end{bmatrix} \dd t=
    \int_0^{- \pi}     {\color{red}\begin{bmatrix}
        0 & 0 \\ 0 & *
    \end{bmatrix}}
    \begin{bmatrix}
        0 \\ 1
    \end{bmatrix} \dd t = \begin{bmatrix}
        0 \\ * 
    \end{bmatrix}.\label{eq:fot2}
\end{align}
Since \eqref{eq:fot1} and \eqref{eq:fot2} have the $*$'s in mutually exclusive coordinates, it must be that the $*$ of \eqref{eq:fot2} is, in fact, $0$, and $\nabla k\left(0, 0\right) = \nabla k\left(0, -\pi\right)$.
Further, we argue that
\[
\nabla k\left(0, 0\right) =
\nabla k\left(0, t\right) =
\nabla k\left(0, -\pi\right),\qquad\forall\,t\in [-\pi,0].
\]
This is because 
\[
    \nabla k (0, t) - \nabla k(0, 0)
    = 
    \begin{bmatrix}
        0 \\ *(t)
    \end{bmatrix}
\]
holds by the same reasoning as before, and the map
\[
t\mapsto *(t) =\frac{\dd}{\dd t}k(te_2) -
\langle e_2,\nabla k(0, 0)\rangle 
,
\]
where $e_2=(0,1)\in\mathbb{R}^2$,
must be non-decreasing by convexity of the function $t\mapsto k(te_2)$.
Since $*(-\pi)=*(0)$, we must have $*(-\pi)=*(t)=*(0)$ for all $t\in [-\pi,0]$.
We can apply this argument for any {\color{red}\textbf {red}} or {\color{blue}\textbf {blue}} segment in \cref{fig:redandblue} and conclude that $\nabla k(x,y)$ is constant on all such segments.

Finally, we note that any point in $\mathbb{R}^2$ is either on a 
{\color{red}\textbf {red}} or {\color{blue}\textbf {blue}} segment, or is enclosed by a rectangle formed by such segments. We conclude that $\nabla k$ is globally constant, i.e., that $k$ is globally affine, by appealing to the following Lemma~\ref{lem:box-constant}. So, $\oS$ is acyclic.

\begin{lemma}
\label{lem:box-constant}
Let $k\colon\mathbb{R}^2\rightarrow\mathbb{R}$ be convex and differentiable.
If $\nabla k$ is constant along the four sides of an axis-aligned rectangle, then $\nabla k $ is also constant throughout the rectangle, including the interior.
\end{lemma}
\begin{proof}
Write $\{(x_0, y_0), (x_1, y_0), (x_0, y_1),(x_1, y_1)\}$ to denote the four corners of the axis-aligned rectangle.
As before, the map
\begin{align*}
    t \mapsto \langle e_i, \nabla k(z+te_i) \rangle = \frac{\dd}{\dd t} k(z+te_i)
\end{align*}
is a non-decreasing function of $t$ by convexity of the function $t \mapsto k(z+te_i)$ for any fixed $z \in \mathbb{R}^{2}$ and unit vectors $e_1 = (1, 0) \in \mathbb{R}^2$ and $e_2 = (0, 1) \in \mathbb{R}^2$.

Then, considering $t \mapsto \langle e_1,\nabla k(x_0 +t(x_1-x_0), y)\rangle$ for $t\in[0,1]$ and $y\in [y_0,y_1]$, the non-decreasing function must in fact be constant because both of the endpoints are on the sides of an axis-aligned rectangle.
That is, $ \langle e_1,\nabla k(x,y)\rangle$ is constant throughout the interior of the rectangle.
By analogous reasoning, we also conclude that $ \langle e_2,\nabla k(x,y)\rangle$ is constant throughout the interior of the rectangle, so $\nabla k$ is constant throughout the interior of the rectangle.
\end{proof}

\paragraph{Case 2.}
Suppose $k\notin C^2$.
Note that $k$, being $2$-smooth by \cref{lem:kissmooth}, is once continuously differentiable, and thus $\nabla k$ is differentiable (Lebesgue) almost everywhere by Rademacher's theorem.
However, we cannot directly use the fundamental theorem of calculus as in Case 1.
Let $\varphi: \R^2 \rightarrow \R$ be a $2$-dimensional rotationally symmetric $C^\infty$ mollifier with compact support of diameter $1$, $\varphi \ge 0$, $\int_{\R^{2}} \varphi = 1$, and $\varphi_{\epsilon} (z) := \epsilon^{-2} \varphi( z / \epsilon )$, satisfying $\varphi_{\epsilon} \rightarrow \delta$ as $\epsilon\rightarrow 0$, where $\delta$ is the Dirac delta function.
Define
\begin{align*}
    k_{\epsilon} (z) &:= k * \varphi_{\epsilon} (z) = \int_{\R^{2}} \varphi_{\epsilon} (u) k (z - u) \, \dd u, \\
    \oS_{\epsilon} (z) &:= \oS * \varphi_{\epsilon} (z) = \int_{\R^{2}} \varphi_{\epsilon} (u) \oS (z - u) \, \dd u, \\
    \mathrm{symm}\big({D\oS_{\epsilon} (z)}\big) &:= \frac{D\oS_{\epsilon}(z) + D\oS_{\epsilon}^{\top}(z)}{2} = \bigopen{\frac{D\oS + D\oS^{\top}}{2}} * \varphi_{\epsilon} (z),
\end{align*}
where $z = (x, y)$.

Since the convolution is linear, $\varphi_{\epsilon} \ge 0$, and $\mathrm{symm}\big(D\oS\big) - \nabla^2 k \succeq 0$ almost everywhere, we have $\mathrm{symm}\big(D\oS_{\epsilon}\big) - \nabla^2 k_{\epsilon} \succeq 0$ everywhere.

Let
\[
\xi(x,y)\coloneqq\lvert \sin x \rvert - \sin x \sin y,
\]
and note that $\xi$ is a $2$-Lipschitz function.
Let $(x, y)$ be a  {\color{red}\textbf {red}} point in  \cref{fig:redandblue}, so it satisfies the condition of \eqref{eq:redlines}.
Then $\xi(x,y)=0$ and
\[
\xi_\epsilon(x,y) \coloneqq \xi * \varphi_\epsilon(x,y) \le 2\epsilon.
\]
From $\mathrm{symm}\big(D\oS_{\epsilon}\big) - \nabla^2 k_{\epsilon} \succeq 0$, or equivalently
\begin{align*}
    \begin{bmatrix}
    \xi_\epsilon(x,y)- \partial_x^2 k_{\epsilon}(x,y)  & - \partial_x \partial_y k_{\epsilon}(x,y)  \\
        - \partial_x \partial_y k_{\epsilon} (x, y) & \xi_\epsilon(y,-x) - \partial_y^2 k_{\epsilon} (x, y)
    \end{bmatrix} \succeq 0,
\end{align*}
we have $0\le \partial_x^2 k_{\epsilon} (x, y) \le 2 \epsilon$ for any {\color{red}\textbf {red}} point $(x, y)$ satisfying \eqref{eq:redlines}.
Considering the other diagonal entry, we have
\begin{align*}
0\le 
\xi_\epsilon(y,-x)- \partial_y^2 k_{\epsilon} (x, y) &\le\xi_\epsilon(y,-x)\le 2.
\end{align*}
Finally, the determinant of the positive semidefinite matrix should be nonnegative, so
\begin{align*}
    |\partial_x \partial_y k_{\epsilon} (x, y)| \le 2\sqrt{\epsilon}.
\end{align*}
Therefore we can observe that for any $(x, y)$ satisfying \eqref{eq:redlines},
\begin{align*}
    \nabla^2 k_{\epsilon} (x, y) & = \begin{bmatrix}
        \gO (\epsilon) & \gO (\epsilon^{1/2}) \\ \gO (\epsilon^{1/2}) & \gO(1)
    \end{bmatrix}
    \to \begin{bmatrix}
        0 & 0 \\ 0 & * 
    \end{bmatrix} \quad \text{as} \ \ \epsilon \to 0.
\end{align*}
(where $\mathcal{O}$ is uniform over all of the {\color{red}\textbf {red}} points).
The same analysis holds for the
{\color{blue}\textbf {blue}} points $(x, y)$ satisfying \eqref{eq:bluelines}.

Now we can replicate the steps of Case 1 with the mollifier.
Consider the line integral across the {\color{red}\textbf {red}} line segment
\begin{align*}
    \nabla k (0, -\pi) - \nabla k (0, 0) 
    &= \lim_{\epsilon \to 0} \bigopen{ \nabla k_{\epsilon} \left( 0, -\pi \right) - \nabla k_{\epsilon} \left( 0, 0 \right)} \\
    &= \lim_{\epsilon \to 0} \int_0^{- \pi} 
    \begin{bmatrix}
        \gO (\epsilon) & \gO (\epsilon^{1/2}) \\ \gO (\epsilon^{1/2}) & \gO(1)
    \end{bmatrix} 
    \begin{bmatrix}
        0 \\ 1
    \end{bmatrix} \dd t = \begin{bmatrix}
        0 \\ * 
    \end{bmatrix},
\end{align*}
where we can evaluate the limit under the integral by the dominated convergence theorem, since $\mathcal{O}(\epsilon)$ and $\mathcal{O}(\sqrt{\epsilon})$ are uniform over the domain of the integral.
Applying this argument for all steps in Case 1 involving the fundamental theorem of calculus (using integrals of the {\color{red}\textbf {red}} or {\color{blue}\textbf {blue}} Hessians), we get the same conclusion as in \eqref{eq:fot1} and \eqref{eq:fot2}.
Therefore $\nabla k$ is constant on the colored points of \cref{fig:redandblue}, and therefore constant on $\mathbb{R}^2$ by \cref{lem:box-constant}.
\hfill$\square$

\end{document}